\documentclass[journal]{IEEEtran}
\IEEEoverridecommandlockouts   

\usepackage{amsmath,amssymb}
\usepackage{graphicx}
\usepackage[square,sort,comma,numbers]{natbib}
\usepackage{epstopdf}
\usepackage[italian,french,english]{babel}

\def\R{\mathcal{R}}

\newtheorem{theorem}{Theorem}
\newtheorem{corollary}[theorem]{Corollary}

\newtheorem{proof}{Proof}
\begin{document}

\title{\LARGE Nonlinear Control of a DC MicroGrid for the Integration of Photovoltaic Panels}
%\title{Management of the Interconnection of Intermittent Photovoltaic Systems Through a DC Link and Storages}
%\author{Alessio, Sabah, Gilney, Abdelkrim, Françoise, Elena, Marika}
%\author{A. Iovine,}
%\author{S. B. Siad,}
%\author{G. Damm,}
%\author{A. Benchaib,}
%\author{F. Lamnabhi-Lagarrigue,}
%\author{E. De Santis,}
%\author{M. D. Di Benedetto}
%\author{A. Iovine$ ^a$\thanks{$ ^a$ Alessio Iovine, Elena De Santis and Marika Di Benedetto are with the Center of Excellence DEWS, Department of Information Engineering, Computer Science and Mathematics, University of L'Aquila, Italy.  Email: alessio.iovine@graduate.univaq.it, \{elena.desantis, mariadomenica.dibenedetto\}@univaq.it. Corresponding author: Alessio Iovine. Permanent email: alessio.iovine@hotmail.com}, G. Damm$ ^b$\thanks{$ ^b$ Gilney Damm and Sabah B. Siad are with IBISC - Universitè d'Evry Val d'Essonne, Evry, France. Email: gilney.damm@ibisc.fr, siadsabah@yahoo.fr}, E. De Santis$ ^a$, M. D. Di Benedetto$ ^a$, A. Benchaib$ ^c$\thanks{$ ^c$ Abdelkrim Benchaib is with Alstom Power, France (Email: abdelkrim.benchaib@alstom.com)}, S. B. Siad$ ^b$}
\author{A. Iovine$ ^a$\thanks{$ ^a$ Alessio Iovine, Elena De Santis and Marika Di Benedetto are with the Center of Excellence DEWS, Department of Information Engineering, Computer Science and Mathematics, University of L'Aquila, Italy.  Email: alessio.iovine@graduate.univaq.it, \{elena.desantis, mariadomenica.dibenedetto\}@univaq.it. Corresponding author: Alessio Iovine. Permanent email: alessio.iovine@hotmail.com}, S. B. Siad$ ^b$, G. Damm$ ^b$\thanks{$ ^b$ Gilney Damm and Sabah B. Siad are with IBISC - Universitè d'Evry Val d'Essonne, Evry, France. Email: gilney.damm@ibisc.fr, siadsabah@yahoo.fr}, E. De Santis$ ^a$, M. D. Di Benedetto$ ^a$ }

\maketitle

%\vspace{-1.50cm}
%\small
\begin{abstract}
%\vspace{-0.3cm}
New connection constraints for the power network (Grid Codes) require more flexible and reliable systems, with robust solutions to cope with uncertainties and intermittence from renewable energy sources (renewables), such as photovoltaic arrays. The interconnection of such renewables with storage systems through a Direct Current (DC) MicroGrid can fulfill these requirements. A "Plug and Play" approach based on the "System of Systems" philosophy using distributed control methodologies is developed in the present work. This approach allows to interconnect a number of elements to a DC MicroGrid as power sources like photovoltaic arrays, storage systems in different time scales like batteries and supercapacitors, and loads like electric vehicles and the main AC grid. The proposed scheme can easily be scalable to a much larger number of elements.
\end{abstract}%
%\vspace{-0.6cm}

\renewcommand{\abstractname}{Note to Practitioners}
%\begin{keyword}%
%\sep \sep \sep \sep
%\end{keyword}%
%\vspace{-0.4cm}
%\section{Note to Practitioners}

\begin{abstract}
%	\vspace{-0.2cm}
Renewable energy can play a key role in producing local, clean and inexhaustible energy to supply the world’s increasing demand for electricity. Photovoltaic conversion of solar energy is a promising solution and is the best fit in several situations. However, its intermittent nature remains a real difficulty that can create instability. To answer to the new constraints of connection to the network (grid codes) for either solar plants than distributed generation, one possible solution is the use of Direct Current (DC) MicroGrids including storage systems, in order to integrate the electric power generated by these photovoltaic arrays. One of the main reasons is the fact that photovoltaic panels, batteries, supercapacitors and electric vehicles are DC. On the other hand, reliable stable control of DC MicroGrids is still an open problem. In particular it lacks rigorous analysis that can establish the operation regions and stability conditions for such MicroGrids. Current works that consider realistic grids usually apply from-the-shelf solutions that do not study the dynamics of such grids. While more rigorous studies just consider too much simplified grids that does not represent the conditions from real life applications. The present work presents nonlinear controllers capable to stabilize the DC MicroGrid, with rigorous analysis on the sufficient conditions and region of operation of the proposed control.
%\vspace{-0.4cm}
\end{abstract}%

\renewcommand{\abstractname}{Keywords}
%\vspace{-0.5cm}
\begin{abstract}
%	\vspace{-0.2cm}
DC power systems, Power generation control, Photovoltaic power systems, Lyapunov methods
\end{abstract}%
\normalsize
%\vspace{-0.5cm}
\section{INTRODUCTION}\label{Sec_Introduction}
%\vspace{-0.3cm}
Renewable energy can play a key role in producing local, clean and inexhaustible energy to supply the world's increasing demand for electricity. Photovoltaic (PV) conversion of solar energy is a promising way to meet the growing demand for energy, and is the best fit in several situations \cite{Eltawil2010PVconnectedgridproblems}. However, its intermittent nature remains a real disability that can create voltage (or even frequency in the case of islanded MicroGrids) instability for large scale grids. In order to answer to the new constraints of connection to the network (Grid Codes) it is possible to consider storage devices \cite{Barton2004energystorage}, \cite{Krajacic20112073}; the whole system will be able to inject the electric power generated by photovoltaic panels to the grid in a controlled and efficient way. As a consequence, it is necessary to develop a strategy for managing energy in relation to the load and the storages constraints. Direct Current (DC) microgrids are attracting interest thanks to their ability to easily integrate modern loads, renewables sources and energy storages \cite{Piagi2004}, \cite{Piagi2006}, \cite{Iravani2007}, \cite{Guerrero2014LVDC}, \cite{Guerrero2013advancedcontrol} since most of them (like electric vehicles, batteries and photovoltaic panels) are naturally DC: therefore, in this paper a DC microgrid composed by a source, a load, two storages working in different time scales, and their connecting devices is considered.

The utilized approach is based on a "Plug and Play" philosophy: the global control will be carried out at local level by each actuator, according to distributed control paradigm. The controller is developed in a distributed way for stabilizing each part of the whole system, while performing power management in real time to satisfy the production objectives while assuring the stability of the interconnection to the main grid.

Even if control techniques for converters are a well known research field \cite{siraramiirez_silva-ortigoza_2006}, \cite{Rantzer2010comparisonHybridConverters}, \cite{A_chen_damm_cdc_2014}, \cite{chen:hal-01159853}, the models generally used assume to have full controllability of the system \cite{A_tahim_pagano_lenz_stramosk_2015}, \cite{A_bidram_davoudi_lewis_guerrero_2013}; in reality, due to technical reasons, systems have an additional variable (a capacitor) on the source side \cite{Walker2004cascadedPVconnection}. Controlling this capacitor implies in leaving another variable (the grid side capacitor) uncontrolled; this remaining dynamics is usually neglected by the assumption that it is connected (and implicitly stabilized) by an always stable strong main grid. Removing this assumption to consider a realistic grid implies that this dynamics needs to be taken into account when studying grid stability. To the best of authors knowledge, no rigourous stability analysis has been developed for a realistic DC MicroGrid. In this paper convergence analysis is performed for all the dynamics composing each device with the target to extract a desired amount of power. In the same way, dynamics interaction is evaluated in order to obtain voltage stability in the DC MicroGrid in response to load and generation variations. Indeed as introduced in this paper, only a rigourous analysis of all the dynamics provide the grid stability conditions to be respected.

The adopted control strategy is shown to work both in case of time-varying uncontrolled load than in case of constant controlled load; both problems being relevant \cite{Marx2012}, \cite{Hamache2014backstepping}. The whole system provides protection against faults and suppresses interference, and has a positive impact on the behaviour of the complete electrical system. The final management system can be configurable and adaptable as needed.

This paper is organized as follows. In Section \ref{Sec_models} the model of the DC MicroGrid is introduced. Then in Section \ref{Sec_control_laws} the adopted control laws for each subsystem are proven to satisfy stability requirements. %Section \ref{Sec_physical_explanation} explains the necessity of the adopted analysis.
Section \ref{sec_simulation_results_energy} provides simulation results about the connected system behaviour, while in Section \ref{sec_conclusion} conclusions are provided.
%
%\vspace{-0.6cm}
\section{Problem definition}\label{Sec_problem_definition}
%\vspace{-0.3cm}
%
The reference framework is depicted in Figure \ref{Fig_microgrid_example}, where the isolated DC microgrid is represented. The targets would be to assure voltage DC grid stability while correctly feeding the load. To each element (PV array, battery, supercapacitor) a DC/DC converter is connected: their dynamical models are described in Section \ref{Sec_models}.
\begin{figure}
  \centering\includegraphics[width=1\columnwidth]{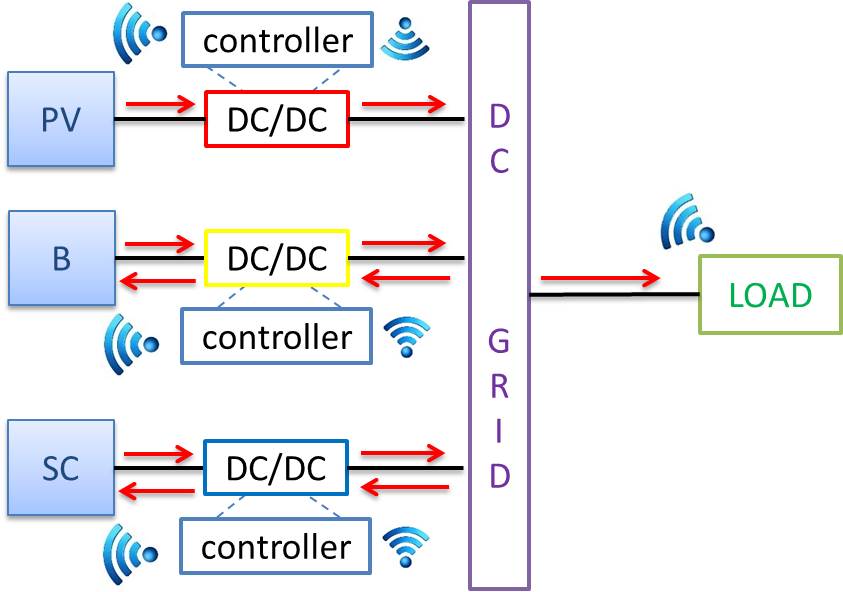}\\
  \caption{The considered framework}\label{Fig_microgrid_example}
%  \vspace{-0.4cm}
\end{figure}
The whole control objective is then split in several tasks; the first one is to extract the maximum available power from the photovoltaic array. This maximum power production is obtained calculating the optimal value of the duty cycle in order to fix the PV array connected capacitor voltage to a given reference. Backstepping theory is used to stabilize the DC/DC boost converter that connects the solar array to the DC network.

The focus then moves to the storage systems and their connection to the DC network. In this paper, two kinds of storage are considered: a battery, which purpose is to provide/absorb the power when needed, and a supercapacitor, which purpose is to stabilize the DC grid voltage in case of disturbances. DC/DC bidirectional converters are necessary to enable the two modes of functioning (charge and discharge). The battery is assimilated as a reservoir which acts as a buffer between the flow requested by the network and the flow supplied by the production sources, and its voltage is controlled by the DC/DC current converter. Again, backstepping theory is used for designing the controller and assuring stability. With this structure, the DC grid is able to provide a continuous supply of good quality energy.

The three converters present in this system must, in a decentralized way, keep the stability of the DC network interconnecting all parts. The final management system can be configurable and adaptable as needed.
%
%
%
%\vspace{-0.3cm}
\subsection{Assumptions}\label{subsec_assumptions}
%  \vspace{-0.2cm}
%
In this paper two main assumptions are made: the first one is the existence of a higher level controller which provide references to be accomplished by the local controllers \cite{Olivares2014Trendsmicrogrid}; the second one is about a proper sizing of each component of the microgrid in order to have feasible power balance.
%
%  \vspace{-0.2cm}
\subsubsection{\textbf{Higher level controller}}\label{SubSubSec_references}
%  \vspace{-0.2cm}
%
The power output coming from the sources needs to be properly coordinated and controlled. Here we assume that a higher level controller provides references for the local controllers: these references change every fixed time interval $T$ and concern the amount of power needed for the next time interval and the desired voltage value for the DC grid. The time interval $T$ is decided by the high level controller according to the computational time needed for calculations. These references are about the desired voltage to impose to the PV array and to the battery to obtain the needed amount of power, $V_1^*$ and $V_4^*$ respectively, and the desired voltage value for the DC grid, $V_9^*$. We need the references to be able to take into account a proper charge/discharge rate power for the supercapacitor; a state of charge about 50\% at the beginning of the time interval is the best starting point for efficiency reasons.
%
%
%  \vspace{-0.2cm}
\subsubsection{\textbf{Energy balance}}\label{SubSubSec_power_balance}
%  \vspace{-0.2cm}
%
Proper sizing of each component in a DC microgrid is an important feasibility requirement. In order to always satisfy the power demanded by the load, the sizing of the PV array, battery and supercapacitor fit some conditions related to the $P_{PV}$ produced power by the photovoltaic array, the $P_B$ and $P_{SC}$ stored power into the battery and the supercapacitor respectively, and the $P_L$ power absorbed by the load:
\begin{itemize}
  \item[i)] the sizing of the photovoltaic array is performed according to total energy needed into a whole day;
\small
\begin{equation}
\int_{0}^{D} P_{PV} \: dt \geq \int_{0}^{D} P_L \: dt
\end{equation}
\normalsize
where $D$ is equal to daytime, 24 hours;
  \item[ii)]
  the sizing of the battery and the supercapacitor are performed according to the energy balance in a $T$ time step , needed for selecting a new reference;
  \footnotesize
\begin{equation}
\left\| \int_{kT}^{(k+1)T} \left(P_{PV} + P_B - P_L  \right) \: dt \right\| \leq \frac{1}{2}\int_{kT}^{(k+1)T} P_{SC} \: dt \:\:\:\: \forall \: k %- \int_{KT}^{(K+1)T} \: dt
\end{equation}
\normalsize
%$\: \forall \: k$.
\end{itemize}

The last condition can be seen as the ability of the supercapacitor to fulfill the request to provide enough amount of power in the considered time interval; for sizing the supercapacitor we consider the worst scenario due to current load variations, i.e. the case where the supercapacitor needs to provide/absorb the maximum available current for all the time step. % where it must be able to take or provide the maximum current from or to the grid.

The complete sizing of these components is considered out of the scope at this point, and will be studied in future works.
%
%
%
%\vspace{-0.6cm}
\section{DC MicroGrid modeling}\label{Sec_models}
%  \vspace{-0.3cm}
In this Section the considered framework depicted in Figure \ref{Fig_microgrid_example} is described. The PV array, battery and
supercapacitor are each one connected to the DC grid by a DC/DC converter. Here the circuital representation and the mathematical
model are given, based on power electronics averaging technique \cite{sanders_noworolski_liu_verghese_1991},
\cite{middlebrook_cuk_1977}.
\begin{equation}
\dot{x}(t) = f(x(t)) + g(x(t),u(t),d(t))+d(t)
\end{equation}
\begin{equation}\label{eq_all_syste_u}
x = \left[\:\: x_{1} \:\:  x_{2} \:\:  x_{3} \:\: x_{4} \:\:  x_{5} \:\:  x_{6} \:\: x_{7} \:\:  x_{8} \:\:  x_{9} \:\: \right]^T
\end{equation}
\begin{equation}\label{eq_all_syste_u}
u = \left[\:\: u_{1} \:\:\:\:  u_{2} \:\:\:\:  u_{3} \:\: \right]^T
\end{equation}
\begin{equation}\label{eq_all_syste_d}
d = \left[\:\:V_{PV}  \:\:\:\:  V_{B} \:\:\:\:  V_{S}  \:\:\:\: \frac{1}{R_L} \:\:\right]^T
\end{equation}
\begin{equation}\label{EQ_DC_microgrid}
\begin{array}{l}
\begin{cases}
\dot{x}_{1}=-\frac{1}{R_{1}C_{1}}x_{1} - \frac{1}{C_{1}}x_{3} + \frac{1}{R_{1}C_{1}}V_{PV}\\ &\mbox{} \\
\dot{x}_{2}=-\frac{1}{R_{2}C_{2}}x_{2} + \frac{1}{C_{2}}x_{3} - \frac{1}{C_{2}}u_{1}x_{3} + \frac{1}{R_{2}C_{2}}x_{9} \\ &\mbox{}
\\
\dot{x}_{3}=\frac{1}{L_{3}}\left[x_{1}-x_{2}-R_{01}x_{3}\right] \\ &\mbox{} \\ \:\:\:\:\:\:\:+\frac{1}{L_{3}}(x_{2}+({R_{01}-R_{02}})x_{3})u_{1}
\\ &\mbox{} \\
\dot{x}_{4}=-\frac{1}{R_{4}C_{4}}x_{4}-\frac{1}{C_{4}}x_{6}+\frac{1}{R_{4}C_{4}}V_{B} \\ &\mbox{} \\
\dot{x}_{5}=-\frac{1}{R_{5}C_{5}}x_{5}+\frac{1}{C_{5}}x_{6}-\frac{1}{C_{5}}u_{2}x_{6}+\frac{1}{R_{5}C_{5}}{x}_{9}  \\ &\mbox{} \\
\dot{x}_{6}=\frac{1}{L_{6}}x_{4}-\frac{1}{L_{6}}x_{5}-\frac{R_{04}}{L_{6}}x_{6}+\frac{1}{L_{6}}x_{5}u_{2}\\ &\mbox{} \\
\dot{x}_{7}=-\frac{1}{R_{7}C_{7}}x_{7}-\frac{1}{C_{7}}x_{8}+\frac{1}{R_{7}C_{7}}{x}_{9}  \\ &\mbox{} \\
\dot{x}_{8}=\frac{1}{L_{8}}V_{S}u_3 -\frac{R_{08}}{L_{8}}x_{8} -\frac{1}{L_{8}} x_7  \\ &\mbox{} \\
 \footnotesize
\dot{x}_{9} = \frac{1}{C_{9}}\left[\frac{x_{2}-{x}_{9}}{R_{2}} + \frac{x_{5}-{x}_{9}}{R_{5}} + \frac{x_{7}-{x}_{9}}{R_{7}} -x_9\frac{1}{{R}_{L}}\right]
\normalsize
 \end{cases}
\end{array}%
\end{equation}
\begin{figure}%[h!]
	\centering
	\includegraphics[width=1\columnwidth]{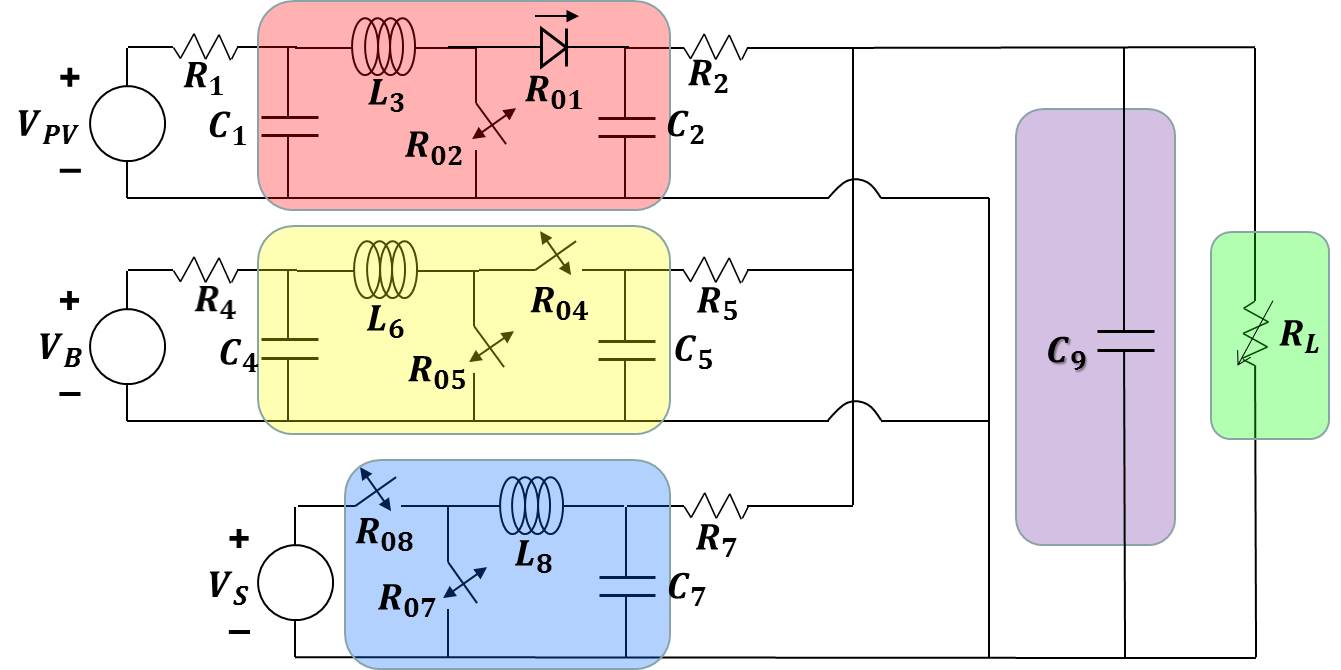}\\
	\caption{The interconnected system.}\label{Fig_interconnected_system}
%	\vspace{-0.6cm}
\end{figure}

%\begin{figure}%[h!]
%%  \centering
%  \includegraphics[width=0.5\columnwidth]{Figures/scheme_with_converters.jpg}
%  \caption{The DC microgrid.}\label{Fig_tecnical_grid_with_converters}
%\end{figure}
Here $x_9$ represents the voltage $V_{C_9}:\R\rightarrow\R^+$ of the capacitor $C_9$, which is the DC microgrid (depicted in violet in Figure \ref{Fig_interconnected_system}. Such voltage is influenced by the connections with load and sources: $R_{L}\in\R^+$ is a constant value representing the load resistance, while the positive values of $R_2$, $R_5$, $R_7$ are the resistances among the dynamics $x_9$ and the interconnected dynamics $x_2$, $x_5$ and $x_7$. These dynamics are the voltages $V_{C_2}:\R\rightarrow\R^+$,  $V_{C_5}:\R\rightarrow\R^+$ and  $V_{C_7}:\R\rightarrow\R^+$ of the capacitors $C_2$, $C_5$ and $C_7$, which are components of the three DC/DC converters connected, respectively, to the PV array (red area), the battery (yellow area) and the supercapacitor (blue area) depicted in Figure \ref{Fig_interconnected_system}. Their description is introduced in the following.
%
%\vspace{-0.4cm}
\subsection{PV branch}
%\vspace{-0.2cm}
%
The DC/DC converter needed to connect the PV array to the DC grid is a boost converter: it is illustrated in the red area in Figure \ref{Fig_interconnected_system}.
The equivalent circuit representation for the boost converter can be expressed using a state space average model. Three state
variables are needed for the system model: the capacitor voltages $V_{C_1}:\R\rightarrow\R^+$ and $V_{C_2}:\R\rightarrow\R^+$ ($x_1$ and $x_2$ respectively) and the inductor current $I_{L_3}:\R\rightarrow\R$ ($x_3$). %The model is written as follows:
$C_1$, $C_2$, $R_1$, $R_2$, $L_{3}$, $R_{01}$, $R_{02}$ are known positive values of the capacitors, resistances and the inductor
while the disturbance $V_{PV}\in\R^+$, is the photovoltaic panel voltage. %$R_{01}$, $R_{02}$ are small
%resistances needed to take into account conduction losses into the switches.
The measured signals are the states $x_1$, $x_2$, $x_3$ and the PV array voltage $V_{PV}$. $u_1$ is the control input, which is defined as the duty cycle of the circuit; its target is to properly integrate the power coming form the PV array and at the same time to obtain the maximum amount of power from the solar energy. This is known as the Maximum Power Point Tracking (MPPT) and consists to regulate the voltage $V_{C_1}$ to its reference $V_1^*=x_1^*$ given by a higher level controller, and considered constant during each time interval $T$.
%
%\vspace{-0.4cm}
\subsection{Battery branch}
%\vspace{-0.2cm}
The DC/DC converter connecting the battery to the DC grid is a bidirectional converter: it is illustrated in the yellow area in Figure \ref{Fig_interconnected_system}. As for the boost converter, we select three state variables: the capacitor voltages $V_{C_4}:\R\rightarrow\R^+$ and $V_{C_5}:\R\rightarrow\R^+$ ($x_4$ and $x_5$ respectively) and the inductor current $I_{L_6}:\R\rightarrow\R$ ($x_6$). $C_4$, $C_5$, $R_4$, $R_5$, $L_{6}$, $R_{06}$ are known positive values of the circuit while the disturbance $V_{B}\in\R^+$ is the battery voltage. The measured signals are the states $x_4$, $x_5$, $x_6$ and the battery voltage $V_{B}$. The duty cycle $u_2$ is the control input; its target is to assign the reference $V_4^*=x_4^*$ value to $x_4$, forcing the battery to provide/absorb a desired amount of power, in a smooth way that maximizes its lifetime. This reference value is given by a higher controller not considered in the present paper, and is considered as constant during the time period $T$.
%
%\vspace{-0.4cm}
\subsection{Supercapacitor branch}
%\vspace{-0.2cm}
%
The DC/DC converter connecting the supercapacitor to the DC grid is a buck one\footnote{The converter used for the supercapacitor has different constraints from the one used for the battery, and for this reason is of a different topology.}. Only two state variables are needed, $x_7$ for the capacitor voltage $V_{C_7}:\R\rightarrow\R^+$ and $x_8$ for the inductor current $I_{L_8}:\R\rightarrow\R$. The positive values of $C_7$, $R_7$, $R_8$, $L_{8}$, $R_{08}$, and the disturbance $V_{S}:\R\rightarrow\R^+$ (supercapacitor voltage) are known at each $t$. The model describing the evolution of the supercapacitor voltage is taken as in \cite{Lifshitz2015Battery}. The duty cycle $u_3$ is the control input for this system. Its target is to control the capacitor voltage directly connected to the grid, which is $x_7$. The measured signals are the states $x_7$, $x_8$, and the supercapacitor voltage $V_{S}(t)>0$.%As already state, in this model the control target will change: indeed, the desired %
%
%\vspace{-0.6cm}
\section{Grid control}\label{Sec_control_laws}
%\vspace{-0.3cm}
In this section control laws are derived to fit the desired targets: for each DC/DC converter a proper control action is developed and the interconnection among all of them is then used to ensure grid stability.
Let us consider the state $x$, whose dynamics are described in (\ref{EQ_DC_microgrid}). Let us consider the set of all possible values of $x_1^*$ that generate a non negative current coming from the PV array as $x_1^*\in\left[\gamma_1 V_{PV}, V_{PV}\right]$, where $\gamma_1=\frac{R_{01}}{R_1}\frac{1}{1+\frac{R_{01}}{R_1}}$. Given a positive value of $R_L$ and $x_9^*$, let us moreover consider the set of all positive values of $x_4^*$ such that the balance of the currents is expressed by
%\small
 \begin{align}\label{eq_current_balance}
\frac{1}{R_L}{x}_{9}^* = \frac{1}{R_{2}}( x_2^*-{x}_{9}^*) + \frac{1}{R_{5}}( x_5^*-{x}_{9}^*)
\end{align}
%\normalsize
where the values of $x_2^*$ and $x_5^*$ are defined as
%
%\begin{align}\label{formula_z2_PV}
%x_2^* &= \frac{x_{9}^*-a_2}{2} + \\
%\nonumber &+\frac{1}{2 a_1} \sqrt{(a_1 x_{9}^* - a_1 a_2)^2 + 4 a_1(\Delta_2 + a_2 x_{9}^*)}
%\end{align}%.
%\small
%\begin{align}\label{formula_z2_PV}
%x_2^* &= \frac{x_{9}^*-a_2}{2} + \\
%\nonumber &+\frac{1}{2 } \sqrt{( x_{9}^* - a_2)^2 + 4 R_{2}C_2(\Delta_2 + a_2 x_{9}^*)}
%\end{align}
%\normalsize%
%\small
\begin{align}\label{formula_z2_PV}
x_2^* = \frac{x_{9}^*-a_2}{2}+\frac{1}{2 } \sqrt{( x_{9}^* - a_2)^2 + 4 R_{2}C_2(\Delta_2 + a_2 x_{9}^*)}
\end{align}
%\normalsize%
%
%\small
\begin{equation}\label{formula_delta_PV}
\Delta_2 = \frac{1}{R_{1}C_2} \left(V_{PV} - x_1^* \right)\left[ x_1^* -  \frac{R_{02}}{R_{1}}(V_{PV} - x_1^*)  \right]
\end{equation}
%\normalsize%.
%
%\begin{equation}\label{formula_a1_PV}
%a_1 = \frac{1}{R_{2}C_2}
%\end{equation}%.
%%
%\small
\begin{equation}\label{formula_a2_PV}
a_2 = \frac{R_{01} - R_{02}}{R_{1}}\left(V_{PV} - x_1^* \right)
\end{equation}
%\normalsize
%
%\begin{equation}\label{formula_z5_BAT}
%x_5^* = \frac{x_{9}^*}{2} + \frac{1}{2 a_1}\sqrt{a_1^2{x_{9}^*}^2 +4 a_1\Delta_5}
%\end{equation}%.
%\small
\begin{equation}\label{formula_z5_BAT}
x_5^* = \frac{x_{9}^*}{2} + \frac{1}{2}\sqrt{{x_{9}^*}^2 +4 R_{5}C_5\Delta_5}
\end{equation}
%\normalsize%.
%
%
%\begin{equation}\label{formula_a1_BAT}
%a_1 = \frac{1}{R_{5}C_5}
%\end{equation}%
%
%and that $V_{DC}$ is such that ${{x}_5^*}^2>-R_{5}C_5\Delta_5$, where
%\small
\begin{equation}\label{formula_delta_BAT}
\Delta_5 = \frac{1}{R_{4}C_5} \left(V_{B}-x_4^* \right)\left[ -x_4^* + R_{04}\frac{1}{R_{4}}(V_{B}-x_4^*)  \right]
\end{equation}
%\normalsize%.
 %, and the fixed given equilibrium points $x_1^*$ respecting (\ref{x_1_star_bounds}), $x_4^*>0$, $x_9^*>0$ given.
and represent the solution of the dynamics $x_2$ and $x_5$ in equation (\ref{EQ_DC_microgrid}) setting $\dot{x}=0$

\begin{theorem}\label{Th_all_syste_stability}
For any given $x_1^*\in\left[\gamma_1 V_{PV}, V_{PV}\right]$, for any given $x_9^*>0$, for any given $x_4^*>0$ such that condition (\ref{eq_current_balance}) is satisfied, under the assumption that for each $t$ the conditions
%
%\small
\begin{equation}\label{eq_all_condition_on_x9}
 x_{2}+(R_{01}-R_{02})x_3 \neq 0, \:\:\:\:\: x_{5} \neq 0, \:\:\:\:\: x_9\neq0
\end{equation}
%\normalsize
%
hold, there exist control laws $u_1$, $u_2$ and $u_3$ such that the closed loop system has an equilibrium point in $x^e$:
\begin{equation}\label{EQ_all_syste_equilibrium_points}
%\vspace{-0.2cm}
x^e = \left[
\begin{array}{c}
x_{1}^e \\
%\alpha_1^e \\
x_{2}^e \\
x_{3}^e \\
%\alpha_3^e \\
x_{4}^e \\
%\alpha_4^e \\
x_{5}^e \\
x_{6}^e \\
%\alpha_6^e \\
x_{7}^e \\
%\alpha_7^e \\
x_{8}^e \\
%\alpha_8^e \\
x_{9}^e %\\
%x_{9}^e%
\end{array}%
\right] =\left[
\begin{array}{c}
x_1^* \\
%\alpha_1 \\
x_2 ^*\\
\frac{1}{R_{1}}(V_{PV}-x_1^*) \\
%\alpha_3 \\
x_4^* \\
%\alpha_4 \\
x_5^* \\
\frac{1}{R_{4}}(V_{B}-x_4^*) \\
%\alpha_6 \\
x_9^* \\
%\alpha_7 \\
0 \\
%\alpha_8 \\
x_9^*
\end{array}
\right].
%\vspace{-0.1cm}
\end{equation}
Moreover, any evolution starting from any (admissible) initial condition $x(0)$ asymptotically converges to it.
\end{theorem}

\begin{proof}
The proof is based on the use of a Lyapunov function $V$ which is a composition of different Lyapunov functions, as illustrated in \cite{kundur2004powerstability}, \cite{B_kundur_balu_lauby_1994}. We use Proportional Integral (PI) control inputs $u_1$ and $u_2$  for properly controlling dynamics $x_1$, $x_3$, $x_4$ and $x_6$ in order to obtain a desired amount of power coming from the PV array and the battery. Then the control input $u_3$ focuses on the grid voltage regulating the interconnection among the systems. The control laws are developed by using a backstepping technique. The proposed Lyapunov function is
%
%\small
\begin{equation}\label{eq_lyap_entire}
V = V_{1,3} + V_{4,6} + V_{7} + V_{8} + V_{2,5,9}
\end{equation}
%\normalsize
%
where all the terms are defined as follows.

Let us first focus on the control $u_1$, which is dedicated to dynamics $x_1$ and $x_3$: it is defined as
%\footnotesize
\small
\begin{align}\label{formula_control_PV}
u_1 &= \frac{1}{{x_{2}+(R_{01}-R_{02})x_3}} \left[-x_{1} + x_{2} + R_{01}x_3 -L_3 v_1 \right]
\end{align}%
\normalsize
with
%
%\footnotesize
\begin{align}\label{formula_control_PV_v1}
v_1 &=  K_3(x_3 - z_3) + \overline{K}_3\alpha_3  - C_1\overline{K}_1K_{1}^\alpha(x_1 - x_1^*)+\\
\nonumber&+\left(C_1K_1 - \frac{1}{R_1} \right)(K_1(x_1 - x_1^*)+ \overline{K}_1\alpha_1)
\end{align}
\begin{equation}\label{formula_eqpoint_x3_1}
z_3 = \frac{1}{R_{1}}(V_{PV} - x_{1}) + C_1K_1(x_1 - x_1^*) + C_1\overline{K}_1\alpha_1
\end{equation}
%\normalsize
%
where the positive gains $K_3$, $\overline{K}_3$, $K_{3}^\alpha$, $\overline{K}_1$, $K_{1}^\alpha$, $K_1$, have to be properly chosen and $\alpha_1$, $\alpha_3$ are integral terms assuring zero error in steady state:
%\small
%\begin{equation}\label{formula_alpha1}
%\dot{\alpha}_1 = K_{1}^\alpha(x_1 - x_1^*)
%\end{equation}
%\normalsize%
%%
%\small
%\begin{equation}\label{formula_alpha3}
%\dot{\alpha}_3 = K_{3}^\alpha(x_3 - z_3)
%\end{equation}
%\normalsize%
%\small
\begin{equation}\label{formula_alpha1}
\dot{\alpha}_1 = K_{1}^\alpha(x_1 - x_1^*) \:\:\:\:\:\:\:\: \dot{\alpha}_3 = K_{3}^\alpha(x_3 - z_3)
\end{equation}
%\normalsize%
%
%\small
%\begin{equation}\label{formula_alpha3}
%
%\end{equation}
%\normalsize%

An augmented system can be considered for the dynamics $x_1$ and $x_3$, where the state, the disturbance vector and the relating matrices are
%
%\begin{equation}\label{EQ_converter_PV_linearizaed_system_state}
%{\bar{x}}_{1,3} = \left[
%\begin{array}{c}
%x_{1} \\
%\alpha_{1} \\
%x_{3} \\
%\alpha_{3}
%\end{array}%
%\right]
%\end{equation}
\begin{equation}\label{EQ_converter_PV_linearizaed_system_state}
{\bar{x}}_{1,3} = \left[\:\: x_{1} \:\: \alpha_{1} \:\: x_{3} \:\: \alpha_{3} \:\:\right]^T
\end{equation}
%
%\begin{equation}\label{EQ_converter_PV_linearizaed_d_state}
%  {\bar{d}}_{1,3} = \left[
%\begin{array}{c}
%V_{PV} \\
%x_{1}^*
%\end{array}%
%\right]
%\end{equation}
%
\begin{equation}\label{EQ_converter_PV_linearizaed_d_state}
  {\bar{d}}_{1,3} = \left[\:\: V_{PV} \:\: x_{1}^* \:\:\right]^T
\end{equation}
\begin{equation}\label{EQ_converter_PV_linearizaed_system}
\dot{\bar{x}}_{1,3}  = A_{1,3}\bar{x}_{1,3} + D_{1,3} {\bar{d}}_{1,3}
\end{equation}
\footnotesize
\begin{equation}\label{eq_PV_augmented_state_x_A13}
A_{1,3} =
\left[
\begin{array}
[c]{cccc}%
-\frac{1}{R_1 C_1} & 0 & -\frac{1}{ C_1} & 0 \\
K_1^\alpha  & 0 & 0 & 0 \\
a_{31} & a_{32} & -K_3 & - \overline{K}_3 \\
K_3^{\alpha}\left(\frac{1}{R_1} - C_1 K_1 \right) & -K_3^{\alpha}C_1 \overline{K}_1 & K_3^{\alpha} & 0
\end{array}
\right]
\end{equation}
\normalsize
%
%\footnotesize
\begin{equation}\label{eq_PV_augmented_state_x_A13a31}
a_{31} = (K_3-K_1)\left( C_1 K_1 - \frac{1}{R_1} \right)  + C_1 \overline{K}_1 K_1^\alpha
\end{equation}
\begin{equation}\label{eq_PV_augmented_state_x_A13a32}
a_{32} = \overline{K}_1 \left( K_3 C_1 - K_1 C_1 + \frac{1}{R_1} \right)
\end{equation}
%\normalsize
%
\begin{equation}\label{eq_PV_augmented_state_x_D13}
D_{1,3} =
\left[
\begin{array}
[c]{cc}%
\frac{1}{R_1 C_1} & 0 \\
0 & -K_1^\alpha  \\
\frac{K_3}{R_1} & d_{32} \\
-\frac{K_3^{\alpha}}{R_1} & C_1 K_1 K_3^{\alpha}
\end{array}
\right]
\end{equation}
%
%\footnotesize
\begin{equation}\label{eq_PV_augmented_state_x_D13d32}
d_{32} = -C_1 \overline{K}_1 K_1^\alpha + K_1\left( C_1 K_1 - K_3 C_1 - \frac{1}{R_1} \right)
\end{equation}
%\normalsize
%
System (\ref{EQ_converter_PV_linearizaed_system}) has the following equilibrium points:
%
%\begin{equation}\label{EQ_converter_PV_linearizaed_system_state_equilibrium}
%  {\bar{x}}_{1,3}^e = \left[
%\begin{array}{c}
%x_{1}^* \\
%0 \\
%\frac{V_{PV}-x_1^*}{R_1}\\
%0
%\end{array}%
%\right]
%\end{equation}
%
%\small
\begin{equation}\label{EQ_converter_PV_linearizaed_system_state_equilibrium}
  {\bar{x}}_{1,3}^e = \left[\:\: x_{1}^* \:\: 0 \:\: \frac{V_{PV}-x_1^*}{R_1} \:\: 0 \:\:\right]^T
\end{equation}
%\normalsize
%
%The calculations are for dynamics $\alpha_1$, $\alpha_3$, $x_1$ and $x_3$, respectively.
The characteristic polynomial is considered; to obtain stability all the terms need to be strictly positive, i.e. $p_3>0$, $p_2>0$, $p_1>0$, $p_0>0$.
%
%\footnotesize
\begin{equation}\label{eq_PV_augmented_state_eigen_pol}
p(\lambda) =  \lambda^4 + p_3 \lambda^3 + p_2 \lambda^2 + p_1 \lambda + p_0
\end{equation}
\begin{equation}\label{eq_PV_augmented_state_eigen_polp3}
p_3 = K_3 + \frac{1}{R_1 C_1} > 0
\end{equation}
\begin{align}\label{eq_PV_augmented_state_eigen_polp2}
p_2 &= \overline{K}_3 K_3^\alpha + \frac{1}{R_1 C_1}K_3 + \overline{K}_1 K_1^\alpha +\\
\nonumber &+ \left[(K_3-K_1)\left(  K_1 - \frac{1}{R_1C_1} \right) \right] > 0
\end{align}
\begin{align}\label{eq_PV_augmented_state_eigen_polp1}
p_1 &= \frac{1}{R_1 C_1}\left(\overline{K}_3 K_3^\alpha + \overline{K}_1 {K}_1^\alpha\right) +\\
\nonumber & + \overline{K}_1 {K}_1^\alpha \left(K_3- K_1 \right) + K_3^\alpha\left( K_1 - \frac{1}{R_1 C_1} \right) > 0
\end{align}
\begin{equation}\label{eq_PV_augmented_state_eigen_polp0}
p_0 = \overline{K}_1 \overline{K}_3  K_3^{\alpha} K_1^{\alpha} > 0
\end{equation}
%
%\normalsize
%
Due to the hypothesis of positive gains, conditions (\ref{eq_PV_augmented_state_eigen_polp3}) and
(\ref{eq_PV_augmented_state_eigen_polp0}) are always satisfied; furthermore, $K_3>K_1$ and $K_1>\frac{1}{R_1 C_1}$ are sufficient
conditions for the (\ref{eq_PV_augmented_state_eigen_polp1}) and (\ref{eq_PV_augmented_state_eigen_polp2}) to be respected. Other conditions are given by the Routh criterium in Table \ref{table_Routh_PV}.

\begin{table}
	% \vspace{-0.2cm}
  \centering
    \caption[Routh table for PV connected converter]{Routh table}\label{table_Routh_PV}
\begin{tabular}{|c|c|c|}
%\hline \textbf{Routh table} \\
\hline 1 & $p_2$ & $p_0$ \\
\hline  $p_3$ & $p_1$ & 0 \\
\hline    $ p_2 - \frac{p_1}{p_3}$ & $p_o p_1 \frac{1}{p_3}$ &  \\
\hline    $ p_2 - p_0 p_1 \frac{1}{p_2 - \frac{p_1}{p_3}}$ & &  \\
\hline
\end{tabular}
%\vspace{0.1cm}
\end{table}
\begin{equation}\label{eq_PV_augmented_state_eigen_polp2_table}
p_2 >\frac{p_1}{p_3}
\end{equation}
\begin{equation}\label{eq_PV_augmented_state_eigen_polp2_table2}
p_2-\frac{p_0p_1}{p_2-\frac{p_1}{p_3}}>0
\end{equation}
It can be shown that these conditions can be fulfilled with a proper choice of the parameter $K_3$. We do not include here the corresponding calculations for lack of space. As a result, an asymptotically stable linear system is obtained; then there will exist a Lyapunov function $V_{1,3}$ in the form of
%
%\small
\begin{equation}\label{Eq_PV_Lyapunov_for_linear}
 V_{1,3} = \frac{1}{2} (\bar{x}_{1,3} - \bar{x}_{1,3}^e)^T P_{1,3} (\bar{x}_{1,3} - \bar{x}_{1,3}^e) > 0
\end{equation}
%\normalsize
%
with
%
%\small
\begin{equation}\label{Eq_PV_Lyapunov_dot_for_linear}
\dot{V}_{1,3} <0
\end{equation}
%\normalsize
%
where the symmetric positive definite matrix $P_{1,3}$ is obtained by the Riccati equation $A_{1,3}^T P_{1,3} + P_{1,3}A_{1,3}=-Q_{1,3}$ such that (\ref{Eq_PV_Lyapunov_dot_for_linear}) is verified.

The control input $u_2$, which is dedicated to control the dynamics $x_4$ and $x_6$ and to ensure a desired charge/discharge behaviour of the battery imposing the reference $x_4^*$, is of the form of
%\footnotesize
\begin{align}\label{formula_control_BAT}
u_{2}&= \frac{1}{x_{5}}\left(-x_{4}+ x_{5} + R_{04}x_{6} + L_6v_2\right)
\end{align}
%\normalsize
%
with
%
%\footnotesize
\begin{align}\label{formula_control_BAT_v2}
v_{2}=& -  K_6(x_6 - z_6) - \overline{K}_6 \alpha_6 + \overline{K}_4K_4^\alpha\left( x_4 -x_4^* \right) +\\
\nonumber & - \left( C_4 K_4 -\frac{1}{R_4}\right)\left( K_4(x_4 - x_4^*) + \overline{K}_4\alpha_4 \right)
\end{align}
%\normalsize
%
where the positive gains $K_6$, $\overline{K}_6$,  $K_{6}^\alpha$, $\overline{K}_4$, $K_{4}^\alpha$, $K_4$, are properly chosen and 
%\footnotesize
\begin{equation}\label{Eq_battery_z6value}
z_6 = \left(\frac{1}{R_{4}}(V_{B}-x_{4}) + C_4K_4(x_4-x_4^*) +C_4\overline{K}_4\alpha_4\right)
\end{equation}
%\normalsize
%
%\small
%\begin{equation}\label{formula_alpha4}
%\dot{\alpha}_4 = K_{4}^\alpha(x_4 - x_4^*)
%\end{equation}
%\normalsize%
%
%\small
\begin{equation}\label{formula_alpha4}
\dot{\alpha}_4 = K_{4}^\alpha(x_4 - x_4^*) \:\:\:\:\:\:\:\:\:\dot{\alpha}_6 = K_{6}^\alpha(x_6 - z_6)
\end{equation}
%\normalsize%
%%
%\small
%\begin{equation}\label{formula_alpha6}
%\dot{\alpha}_6 = K_{6}^\alpha(x_6 - z_6)
%\end{equation}
%\normalsize%
%
%with $\overline{K}_4>0$, $K_{4}^\alpha >0$ and $K_{6}^\alpha >0$.
with $\alpha_4$ and $\alpha_6$ being integral terms assuring zero error in steady state.

As for the previous case, an augmented system can be considered:
%As for the previous case, the control in closed loop describe an augmented linear system for the augmented state ${\bar{x}}_{4,6}$ and disturbance vector ${\bar{d}}_{4,6}$
%
\begin{equation}\label{EQ_converter_BAT_linearizaed_system}
 \dot{\bar{x}}_{4,6}  = A_{4,6}\bar{x}_{4,6} + D_{4,6} {\bar{d}}_{4,6}
\end{equation}
%
%
%\begin{equation}\label{EQ_converter_BAT_linearizaed_system_state}
%  {\bar{x}}_{4,6} = \left[
%\begin{array}{c}
%x_{4} \\
%\alpha_{4} \\
%x_{6} \\
%\alpha_{6}
%\end{array}%
%\right]
%\end{equation}
%
\begin{equation}\label{EQ_converter_BAT_linearizaed_system_state}
  {\bar{x}}_{4,6} = \left[\: \: x_{4} \: \: \alpha_{4} \: \: x_{6} \: \: \alpha_{6} \: \:\right]^T
\end{equation}
\begin{equation}\label{EQ_converter_BAT_linearizaed_d_state}
  {\bar{d}}_{4,6} = \left[\: \: V_{B} \: \: x_{4}^* \: \:\right]^T
\end{equation}
%
%
%\begin{equation}\label{EQ_converter_BAT_linearizaed_d_state}
%  {\bar{d}}_{4,6} = \left[
%\begin{array}{c}
%V_{B} \\
%x_{4}^*
%\end{array}%
%\right]
%\end{equation}
%
The system in (\ref{EQ_converter_BAT_linearizaed_system}) has matrices $A_{4,6}$ and $D_{4,6}$ similar to $A_{1,3}$ and $D_{1,3}$ in (\ref{EQ_converter_PV_linearizaed_system}), with respect to the considered gains. The same considerations drive to the same asymptotic stability result; then there will exist a Lyapunov function $V_{4,6}$ in the form of
%
%\small
\begin{equation}\label{Eq_BAT_Lyapunov_for_linear}
 V_{4,6} = \frac{1}{2} (\bar{x}_{4,6} - \bar{x}_{4,6}^e)^T P_{4,6} (\bar{x}_{4,6} - \bar{x}_{4,6}^e) > 0
\end{equation}
%\normalsize
with
%
%\small
\begin{equation}\label{Eq_BAT_Lyapunov_dot_for_linear}
 \dot{V}_{4,6} <0
\end{equation}
%\normalsize
%
where the symmetric positive definite matrix $P_{4,6}$ is obtained by the Riccati equation $A_{4,6}^T P_{4,6} + P_{4,6}A_{4,6}=-Q_{4,6}$ such that (\ref{Eq_BAT_Lyapunov_dot_for_linear}) is verified.

%They are proven to be stable by the Lyapunov functions $V_{1,3}$ and $V_{4,6}$ in (\ref{Eq_PV_Lyapunov_for_linear}) and (\ref{Eq_BAT_Lyapunov_for_linear}) and their time derivatives (\ref{Eq_PV_Lyapunov_dot_for_linear}) and (\ref{Eq_BAT_Lyapunov_dot_for_linear}).
Let us now focus on the control input $u_3$, which is determined to ensure voltage grid stability. It does not act directly. It does not act directly on the DC grid, but through the dynamics $x_8$ and $x_7$. Utilizing backstepping and Lyapunov methods, we can select the desired value for the dynamics to control the grid. The Lyapunov functions provided at each step will be used for the entire system, thereby leading to the study of
composite Lyapunov functions.

%Let us now focus on the control input $u_3$, which is calculated to ensure voltage grid stability and to assign a desired transient with respect to the disturbance, which is mainly the load but also the convergence time needed by the controllers $u_1$ and $u_2$ to reach the steady state points. The controller $u_3$ does not act directly on the DC grid, but acts through the dynamics $x_8$ and $x_7$. The dynamics $x_7$ can then be seen as the control input in the DC grid: we can impose a trajectory $z_7(x(t))$, which is function of the state, such to avoid large voltage variations. Utilizing backstepping and Lyapunov methods, we can select the desired value for the $x_{7}$ variable. The Lyapunov functions provided for each steps will be used for the entire system treatment: at the end the result is a study of composite Lyapunov functions.
%%: we can then have a degree of freedom on the choice of the value of $x_7$. The resulting control input $u_3$ will then be calculated.

%The Lyapunov function $V$ is defined as
%%
%\begin{equation}\label{eq_lyap_entire}
%V = V_{1,3} + V_{4,6} + V_{7} + V_{8} + V_{2,5,9}
%\end{equation}
%%
%where $V_{1,3}$ and $V_{4,6}$ have already been introduced and $V_{7}$, $V_{8}$ and $ V_{2,5,9}$ must be found.
The function $V_{2,5,9}$ refers to dynamics $x_2$, $x_5$ and $x_9$; introducing the errors $e_2$ and $e_5$ between the dynamics and their equilibrium points as
\begin{equation}\label{eq_error2}
\nonumber  e_2 = x_2 - x_2^*, \: \: \: \: \: \: \: \:   e_5 = x_5 - x_5^*
\end{equation}
%
%\begin{equation}\label{eq_error5}
%  e_5 = x_5 - x_5^*
%\end{equation}
%
we can rewrite the equations as
%
%\footnotesize
\begin{equation}\label{}
\begin{array}{l}
\begin{cases}
\dot{e}_{2}=\frac{1}{R_{2}C_{2}}(x_9 - e_2 - x_2^*) + \frac{1}{C_{2}}x_{3}(1-u_{1})  \\ &\mbox{}
\\
\dot{e}_{5}=\frac{1}{R_{5}C_{5}}(x_9 - e_5 - x_5^*)+\frac{1}{C_{5}}x_{6}(1-u_{2})  \\ &\mbox{} \\
\dot{x}_{9} = \frac{1}{C_{9}}\left(\frac{1}{R_{2}}(e_2 + x_2^*-{x}_{9}) + \frac{1}{R_{5}}(e_5 + x_5^*-{x}_{9})\right) + \\&\mbox{} \\ +\frac{1}{C_{9}}\left(\frac{1}{R_{7}}(x_{7}-{x}_{9}) -\frac{1}{R_L }{x}_{9}\right) %\\ &\mbox{} \\
 \end{cases}
\end{array}%
\end{equation}
%\normalsize
%
%where ${\alpha}_9$ is the integral error term for the dynamics $x_9$ with respect to its desired equilibrium point.
To find a proper controller, the $V_{2,5,9}$ can be defined as
%
%\small
\begin{equation}\label{Eq_interconnected_Lyapunov}
  V_{2,5,9} = \frac{C_2}{2}e_2^2 + \frac{C_5}{2}e_5^2 + \frac{C_9}{2}x_9^2 %+ \frac{1}{2}\alpha_9^2
\end{equation}
%\normalsize
%
%\begin{equation}\label{}
%  \dot{V}_{2,5,9} = {C_2}e_2\dot{e}_2 + {C_5}e_5\dot{e}_5+ {C_9}x_9\dot{x}_9% + \alpha_9\dot{\alpha}_9
%\end{equation}
%
Then,
%
%\footnotesize
\begin{align}\label{eq_Lyap_259_dot_}
  \dot{V}_{2,5,9} &= -\frac{1}{R_{2}}e_2^2 + e_2\left( \frac{1}{R_{2}}(x_9 - x_2^*) + x_{3}(1-u_{1})\right) + \\
  \nonumber &-\frac{1}{R_{5}}e_5^2 + e_5\left(\frac{1}{R_{5}}(x_9 - x_5^*)+ x_{6}(1-u_{2})\right)\\
  \nonumber &+ x_9\left(\frac{1}{R_{2}}(e_2 + x_2^*-{x}_{9}) + \frac{1}{R_{5}}(e_5 + x_5^*-{x}_{9})\right) + \\ \nonumber &+x_9\left(\frac{1}{R_{7}}(x_{7}-{x}_{9}) -\frac{1}{R_L}{x}_{9}\right)%\\
 % \nonumber & + \alpha_9 K_9^\alpha(x_9-x_9^*)
\end{align}
%\normalsize
%
%
%
In equation (\ref{eq_Lyap_259_dot_}) the dynamics $x_7$ can be seen as control input; it can be properly chosen to obtain a desired form for the $\dot{V}_{2,5,9}$. By choosing the value $z_7$ for $x_7$ as
%
%\footnotesize
\begin{align}\label{Eq_z7}
  z_7 &= - R_7 \frac{1}{{x}_{9}}\left[e_2\left( \frac{1}{R_{2}}(x_9 - x_2^*) + x_{3}(1-u_{1})\right)\right] +\\
  \nonumber & - R_7 \frac{1}{{x}_{9}}\left[e_5\left(\frac{1}{R_{5}}(x_9 - x_5^*)+ x_{6}(1-u_{2})\right) \right] + \\
  \nonumber & +R_7 \left[\frac{{x}_{9}}{R_L} - \frac{1}{R_{2}}(e_2 + x_2^*-{x}_{9})- \frac{1}{R_{5}}(e_5 + x_5^*-{x}_{9}) \right] + \\
  \nonumber & + \frac{1}{{x}_{9}}\left(-{x_9^*}^2 + 2 x_9x_9^*\right) %- R_7\frac{1}{{x}_{9}}(\alpha_9 K_9^\alpha(x_9-x_9^*) + \alpha_9^2)
\end{align}
%\normalsize
%
the $\dot{V}_{2,5,9}$ results to be semidefinite negative. To prove asymptotic stability the set $\Omega$ is considered: it is the largest invariant set of the set $E$ of all points where the Lyapunov function is not decreasing. $\Omega$ contains an unique point; then applying LaSalle's theorem asymptotic stability of the equilibrium point can be established.
%
%\small
\begin{align}\label{Eq_interconnected_Lyapunov_dot}
  \dot{V}_{2,5,9} &= -\frac{1}{R_{2}}e_2^2 -\frac{1}{R_{5}}e_5^2 -\frac{1}{R_{7}}(x_9-x_9^*)^2\leq 0% - \alpha_9^2
\end{align}
%\normalsize
%
%
%\footnotesize
\begin{align}\label{eq_omega _forsemidef}
  \Omega =& \{(e_2,e_5,x_9) \: : \: x_2=x_2^*, \: x_5=x_5^*, \: x_9=x_9^* \} = \\
\nonumber  =& \{(0,0,x_9^*)\}
\end{align}
%\normalsize
%To assign the reference $z_7$ to $x_7$ a control input $u_3$ is calculated:
%

In order to calculate $z_8(t)$ such that $x_8$ backsteps the value of $x_7$ to the desired value $z_7$, we use the Lyapunov function
%\small
\begin{equation}\label{EQ_Lyapunov_SC7}
V_{7} = \frac{1}{2}(x_7 - z_7)^2% + \frac{\overline{K}_7}{2K_7^{\alpha}} \alpha_7^2
\end{equation}
%\normalsize%
where the desired dynamics for $x_7$ is
%\small
\begin{equation}\label{EQ_SC_x7_mod}
\dot{x}_{7} = -{K_7}(x_7 - z_7)% - \overline{K}_7\alpha_7
\end{equation}
%\normalsize
%
with a gain $K_7>0$. By Lyapunov function time derivative calculation, we obtain the reference $z_8$:
%
%\footnotesize
\begin{equation}\label{eq_z8}
z_{8} =  C_{7}K_7(x_7 - z_7) - \frac{1}{R_{7}}(x_{7}-x_9) - C_{7}\dot{z}_7 %+ C_7 \overline{K}_7 \alpha_7
\end{equation}
%\normalsize
%
%The desired dynamics for $x_7$ is then
%\begin{equation}\label{formula_controlled_x7}
%\dot{x}_7 = - K_7(x_7 - z_7)% - \overline{K}_7\alpha_7
%\end{equation}
%where $\alpha_7$ is introduced in (\ref{formula_alpha7})
%
%
Indeed the Lyapunov derivative is negative definite if the value of $x_8$ is properly chosen as $z_8$ in (\ref{eq_z8}):
%
%\footnotesize
\begin{align}\label{EQ_Lyapunov_SC7_}
\dot{V}_{7} &= (x_7 - z_7)(\dot{x}_7 - \dot{z}_7) = \\
\nonumber & = (x_7 - z_7)\left(-\frac{1}{R_{7}C_{7}}x_{7}-\frac{1}{C_{7}}x_{8}+\frac{1}{R_{7}C_{7}}{x}_{9} - \dot{z}_7\right)
\end{align}
%\normalsize
%
%
%\small
\begin{equation}\label{EQ_Lyapunov_SC7_dot}
\dot{V}_{7} = -K_7(x_7 - z_7)^2 < 0
\end{equation}
%\normalsize
%
To calculate the control input $u_3$ we again use backstepping technique: we can obtain it by using the Lyapunov function
%
%\small
\begin{equation}\label{EQ_Lyapunov_SC8}
V_{8} =  \frac{1}{2}(x_8 - z_8)^2 %+ \frac{\overline{K}_8}{2K_8^{\alpha}} \alpha_8^2
\end{equation}
%\normalsize%
%
%\begin{equation}\label{formula_alpha8}
%\dot{\alpha}_8 = K_{8}^\alpha(x_8 - z_8)
%\end{equation}%
whose time derivative is
%\footnotesize
\begin{align}\label{EQ_Lyapunov_SC8_}
\dot{V}_{8} &=  (x_8 - z_8)(\dot{x}_8 - \dot{z}_8)  = \\
\nonumber & =(x_8 - z_8)\left(\frac{1}{L_{8}}V_{S}u_3 -\frac{R_{08}}{L_{8}}x_{8} -\frac{1}{L_{8}} x_7 - \dot{z}_8\right) %+ \frac{\overline{K}_8}{2K_8^{\alpha}} \alpha_8^2
\end{align}
%\normalsize
%

To have a definite positive $\dot{V}_8$ the control input must be
%\small
\begin{align}\label{formula_control_SCtheo}
u_3 =& \frac{1}{V_{S}}\left[x_7 + R_{08}x_8 + L_8 \dot{z}_8 - L_8 K_8(x_8 - z_8)\right]
\end{align}
%\normalsize
%
with $K_8>0$ and constant,
%
%
%\small
\begin{equation}\label{}
\dot{z}_{8} =  C_{7}K_7(\dot{x}_7 - \dot{z}_7) - \frac{1}{R_{7}}(\dot{x}_{7}-\dot{x}_9) - C_{7} \ddot{z}_7%+ C_7 \overline{K}_7 \dot{\alpha}_7
\end{equation}
%\normalsize
%
and where the function $\ddot{z}_7$ is the time derivative of $\dot{z}_7$.
%
%\small
\begin{align}\label{EQ_Lyapunov_SC8_dot}
\dot{V}_{8} &= (x_8 - z_8)(\dot{x}_8 - \dot{z}_8) = -K_8(x_8 - z_8)^2< 0
\end{align}
%\normalsize%
%
Then, using the control laws defined in (\ref{formula_control_PV}), (\ref{formula_control_BAT}) and (\ref{formula_control_SCtheo}), in accordance with (\ref{Eq_PV_Lyapunov_for_linear}), (\ref{Eq_BAT_Lyapunov_for_linear}), (\ref{Eq_interconnected_Lyapunov}), (\ref{EQ_Lyapunov_SC7}) and (\ref{EQ_Lyapunov_SC8}), we have positive definite Lyapunov functions $V_{1,3}>0$, $V_{4,6}>0$, $V_{2,5,9}>0$, $V_7>0$, $V_8>0$ such that their time derivatives $\dot{V}_{1,3}<0$, $\dot{V}_{4,6}<0$, $\dot{V}_{2,5,9}\leq0$, $\dot{V}_7<0$, $\dot{V}_8<0$ ensure stability, according to (\ref{Eq_PV_Lyapunov_dot_for_linear}), (\ref{Eq_BAT_Lyapunov_dot_for_linear}), (\ref{Eq_interconnected_Lyapunov_dot}), (\ref{EQ_Lyapunov_SC7_dot}) and (\ref{EQ_Lyapunov_SC8_dot}).% The role of the gains $K_7$ and $K_8$ is to impose a faster transient to the dynamics of $x_7$ and $x_8$, that are the ones in charge to ensure voltage stability. %Then $V=V_{259}+V_7+V_8>0$ and $\dot{V}\leq0$ ensure stability.%

The composite positive definite Lyapunov function $V$ in (\ref{eq_lyap_entire}) results to have a negative semidefinite time derivative $\dot{V}$; LaSalle's theorem ensures asymptotic stability of the equilibrium $x^e$ of entire system describing the DC microgrid:
%
%\small
\begin{equation}\label{}
  \dot{V} = \dot{V}_{1,3} + \dot{V}_{4,6} + \dot{V}_{7} + \dot{V}_{8}+ \dot{V}_{2,5,9} \leq 0
\end{equation}
%\normalsize
%
%The equilibrium point is $x^e$.
%
\end{proof}
%
%
%A energy function of the system is developed and used to select an adequate reference for the variable which works as a control
%input, $z_7$.

As proven in the previous theorem, the unconstrained control laws $u_1$, $u_2$ and $u_3$ solve our problem. When considering a realistic application, control laws must be bounded: $u_1\in[0,1]$, $u_2\in[0,1]$ and $u_3\in[0,1]$. These bounds also impose limitations for $x_1^*$, $x_4^*$, $x_9^*$. Bounds on $x_1^*$ have already been considered in the theorem to have a current coming from the PV array. A bounded $u_2$ imposes bounds on $x_{9}^*$, $x_{9}^*\in(max(V_{PV},V_B),V_{S})$, and on $x_4^*$: $x_{4}^{\ast}\in\left[ \gamma_2 V_{B}, \beta(x_{9}^*,V_B)\right]  $, where $\gamma_2 =  \frac{R_{04}}{R_4}\frac{1}{1+\frac{R_{04}}{R_4}}$ and
\begin{equation}\label{}
\beta(x_{9}^*,V_B) = \frac{x_{9}^* + \left(\frac{R_5}{R_4}-\frac{R_{04}}{R_4} \right)V_B}{1+\frac{R_5}{R_4}-\frac{R_{04}}{R_4}}
\end{equation}
When considering the bounds $u_1\in[0,1]$ and $u_2\in[0,1]$, bounds on the quantity ${R_L}$ must be satisfied as well: indeed, given a $x_9^*>0$, only the values of ${R_L}\in\Omega_{R_L}$, where
%
%\small
\begin{align}\label{}
\Omega_{R_L}&= \{ {R_L} : x_4^*\in[\gamma_2V_{B},\beta(x_{9}^*,V_{B})] \\
\nonumber& \text{ for some }x_1^*\in[\gamma_1V_{PV},V_{PV}] \}
\end{align}
%\normalsize
%
is the set such that the condition (\ref{eq_current_balance}) is satisfied with respect to physical limitation of all the components of the circuits.

%Consequently bounds on $R_L$ are considered too: ${R_L}\in\Omega_{R_L}$
%\begin{equation}\label{}
%\Omega_{R_L}= \{ {R_L} : x_4^*\in[\gamma_2V_{B},\beta(V_{B},x_{9}^*)] \text{ for some }x_1^*\in[\gamma_1V_{PV},V_{PV}] \}
%\end{equation}
%such that
%
%, ,
Let $X=\R^{9}$ be the state space of the closed loop system. Given the state feedback control laws $u_{1}:X\rightarrow \R$, $u_{2}:X\rightarrow \R$, $u_{3}:X\rightarrow \R$, in (\ref{formula_control_PV}), (\ref{formula_control_BAT}), (\ref{formula_control_SCtheo}), for any value of the used gains we need to compute the maximal set
%\small
\begin{equation}\label{eq_set_omegaK}
  \Omega_{K}\subset X
\end{equation}
%\normalsize
which is invariant for the closed loop dynamical system, and is such that $u_{1}(z)\in\left[  0,1\right]  $, $u_{2}(z)\in\left[  0,1\right]  $, $u_{3}(z)\in\left[  0,1\right]  $, $\forall z\in\Omega_{K}$. Such a maximal set is well defined, because the family of all invariant set in $X$ is closed under union. The set $\Omega_{K}$ is hard to describe, but we can ensure that it is not empty.
\begin{theorem}
For any ${R_L}\in\Omega_{R_L}$, $\forall$ $x_{1}^{\ast}\in\left[\gamma_1 V_{PV}, V_{PV}\right]$ : $x_{4}^{\ast}\in\left[ \gamma_2 V_{B}, \beta(x_{9}^*,V_B)\right]  $, $\forall$ $x_9^*$ such that $max(V_{PV},V_B)<{x}_{9}^*<V_{S}$ there exist gains $K_{1}$, $\overline{K}_{1}$, $K_{1}^\alpha$, $K_{3}$, $\overline{K}_{3}$, $K_{3}^\alpha$, $K_{4}$, $\overline{K}_{4}$, $K_{4}^\alpha$, $K_{6}$, $\overline{K}_{6}$, $K_{6}^\alpha$, $K_7$, $K_8$ such that $\Omega_{K}\neq\emptyset$.
\end{theorem}
\begin{proof}
To prove controller existence, let us now consider the controllers $u_1$ and $u_2$ where no integral error correction terms are considered: the value of the gains $\overline{K}_{1}$, $K_{1}^\alpha$, $\overline{K}_{3}$, $K_{3}^\alpha$, $\overline{K}_{4}$, $K_{4}^\alpha$, $\overline{K}_{6}$, $K_{6}^\alpha$, is set to be zero. As we are neglecting the integral error dynamics, stability conditions to be respected are stated only by the formulas in (\ref{eq_PV_augmented_state_eigen_polp3}) and (\ref{eq_PV_augmented_state_eigen_polp2}). A choice of $K_{1}$ and $K_{3}$ ($K_{4}$ and $K_{6}$) respecting these conditions is done: $K_1 = \frac{1}{R_1 C_1}$ and $K_3 = \frac{R_{01}}{L_3}$ ($K_4 = \frac{1}{R_4 C_4}$ and $K_6 = \frac{R_{04}}{L_6}$). A simple choice of $K_{7}$ and $K_{8}$ respecting stability is $K_7 = \frac{1}{R_7 C_7}$ and $K_8 = \frac{R_{08}}{L_8}$.
%
%For a choice of $K_{1}$ and $K_{3}$ respecting these conditions, t
The conditions $x_{1}^{\ast}\in\left[  \gamma_1 V_{PV},V_{PV}\right]  $, $x_{4}^{\ast}\in\left[ \gamma_2 V_{B}, \beta(x_{9}^*,V_B)\right]  $, ${R_L}\in\Omega_{R_L}$, $max(V_{PV},V_B)<{x}_{9}^*<V_{S}$  implies that $x^{e}$ is an equilibrium with $u_{1}(x^{e})\in\left[  0,1\right]  $, $u_{2}(x^{e})\in\left[  0,1\right]  $, $u_{3}(x^{e})\in\left[  0,1\right]  $. Therefore $x^{e}\in\Omega_{K}$. %The simplest choice is then $K_1 = \frac{1}{R_1 C_1}$ and $K_3 = \frac{R_{02}}{L_3}$.
\end{proof}
\begin{corollary}
For any given ${R_L}\in\Omega_{R_L}$, $\forall$ $x_{1}^{\ast}\in\left[\gamma_1 V_{PV}, V_{PV}\right]$ : $x_{4}^{\ast}\in\left[ \gamma_2 V_{B}, \beta(x_{9}^*,V_B)\right]  $, $\forall$ $x_9^*$ such that $max(V_{PV},V_B)<{x}_{9}^*<V_{S}$, given the state feedback control laws $u_1$, $u_2$, $u_3$, defined in (\ref{formula_control_PV}), (\ref{formula_control_BAT}), (\ref{formula_control_SCtheo}), and for any initial state in $\Omega_{K}$, the state $z(t)$ of the closed loop system asymptotically converges to $x^{e}$, and $u_{1}(z(t))\in\left[0,1\right]  $, $u_{2}(z(t))\in\left[0,1\right]  $, $u_{3}(z(t))\in\left[0,1\right]  $, $\forall t\geq0$.
\end{corollary}
%
%
%
%\vspace{-0.8cm}
\section{SIMULATION RESULTS}\label{sec_simulation_results_energy}
%\vspace{-0.3cm}
In this section we present some simulations that show the results obtained using the proposed control inputs. Such simulations are obtained
using Matlab. The values of the parameters for the model are depicted in Tables \ref{table_PV}. %, \ref{table_BAT} and \ref{table_SC}.%, while the initial states are in Table \ref{table_initial_values_simulation_DCgrid}.
%
%\begin{table}
%  \centering
%\begin{tabular}{|c|c|c|c|}
%%\hline \textbf{State} & \textbf{ $x_{1}(0)$} & \textbf{$x_{2}(0)$} & \textbf{$x_{3}(0)$}\\
%\hline $x_{1}$(0) & 300 $V$ & $x_{7}$(0) & 1000 $V$ \\
%\hline $x_{2}$(0) & 1000 $V$ & $x_{8}$(0) & 0 $A$ \\
%\hline $x_{3}$(0) & 0 $A$ & $x_{9}$(0) & 1000 $V$ \\
%\hline $x_{4}$(0) & 400 $V$ & $V_{PV}$ & 300 $V$ \\
%\hline $x_{5}$(0) & 1000 $V$ & $V_{B}$ & 400 $V$ \\
%\hline $x_{6}$(0) & 0 $A$ & $V_{S}$ & 1900 $V$ \\
%\hline
%\end{tabular}
%  \caption[Microgrid initial values]{Initial values used.}\label{table_initial_values_simulation_DCgrid}
%\end{table}
%%
The simulation target is to correctly feed a load and to maintain the grid stability, which means to ensure no large variation in the DC grid voltage. The simulation time is twenty seconds. The considered reference value $x_9^*$ for the DC grid voltage is selected as $x_9^*=1000$ $V$. A secondary controller is supposed to provide the references to be reached each time interval; during that period the introduced control laws will bring the devices to operate in the desired points. The selected strategy assign to the PV and battery the duty to fulfill the losses into the network; the references will then be calculated according to the load current. The references are updated every second: during the first ten seconds the load is supposed to not vary during each time interval of a second. %the load current is supposed constant; then a step variation in order to prove robustness at step disturbances is introduced (see Figure \ref{Fig_load SC_currents}) during the first ten seconds of the simulation. Moreover, in addition to the step variation, the load is supposed to be time varying in the remaining part of the simulation.; then the

We can split the simulation in two parts: in the first one, which is from zero to ten seconds, the voltages of the PV array and of the battery are constant and the load resistance piecewise constant. Furthermore, the references provided by the higher level controller are exact and the supercapacitor is needed only for providing grid stability during the transient time needed by the converters that are connected to the PV and the battery. In the second part of the simulation, in addition to the step variation, the load is supposed to be time varying and disturbances acting on the PV voltage are considered (see Figure \ref{Fig_PV_BAT_SC_voltagesVpvVb}). To better represent any possible case, we have also simulated the case where the references do not fulfill the energy balance: there the supercapacitor will need to provide power during all the time window. Figure \ref{Fig_PV_BAT_SC_currentsx8iL} describes the values of the resistance $\frac{1}{R_L}$ over the time; we can see the described behaviour, which is piecewise constant for ten seconds and then becomes time-varying.
%
%\begin{table}
%  \centering
%    \begin{tabular}{ | l | l | l | l |}
%    \hline
%    Parameter  & Value & Parameter  & Value \\ \hline
%    $C_1$  & 0.1 F & $L_{3}$ & 0.033 H \\ \hline
%    $C_2$ & 0.01 F & $R_{01}$ & 0.01 $\Omega$ \\ \hline
%    $R_1$ & 0.1 $\Omega$ & $R_{02}$ & 0.01 $\Omega$ \\ \hline
%    $R_2$ & 0.1 $\Omega$ &    &  \\
%    \hline
%    \end{tabular}
%  \caption{PV connected converter parameters.}\label{table_PV}
%\end{table}
%
\begin{table}
  \centering
      % \vspace{-0.3cm}
       \caption{Grid parameters.}\label{table_PV}
    \begin{tabular}{ | l | l | l | l |}
    \hline
    Parameter  & Value & Parameter  & Value \\ \hline
    $C_1$  & 0.1 F & $L_{3}$ & 0.033 H \\ \hline
    $C_2$ & 0.01 F & $R_{01}$ & 0.01 $\Omega$ \\ \hline
    $R_1$ & 0.1 $\Omega$ & $R_{02}$ & 0.01 $\Omega$ \\ \hline
    $R_2$ & 0.1 $\Omega$ & $C_4$   & 0.1 F \\ \hline
   %     & & & \\ \hline
    $C_5$ & 0.01 F & $R_{04}$ & 0.01 $\Omega$ \\ \hline
    $R_4$ & 0.1 $\Omega$ & $R_{05}$ & 0.01 $\Omega$ \\ \hline
    $R_5$ & 0.01 $\Omega$ & $L_{6}$   & 0.033 H \\ \hline
        $C_7$  & 0.01 F & $L_{8}$ & 0.0033 H \\ \hline
    $R_{07}$ & 0.01 $\Omega$  & $R_{08}$ & 0.01 $\Omega$ \\ \hline
 %   $R_7$ & 0.1 $\Omega$ & $R_{05}$ & 0.001 $\Omega$ \\ \hline
    $R_7$ & 0.1 $\Omega$ & $C_9$  & 0.0001 F  \\
    \hline
    \end{tabular}

%  \vspace{-0.7cm}
\end{table}
%\vspace{-0.4cm}
%
%\begin{table}
%  \centering
%    \begin{tabular}{ | l | l | l | l |}
%    \hline
%    Parameter  & Value & Parameter  & Value \\ \hline
%    $C_4$  & 0.1 F & $L_{6}$ & 0.033 H \\ \hline
%    $C_5$ & 0.01 F & $R_{04}$ & 0.01 $\Omega$ \\ \hline
%    $R_4$ & 0.1 $\Omega$ & $R_{05}$ & 0.01 $\Omega$ \\ \hline
%    $R_5$ & 0.01 $\Omega$ &    &  \\
%    \hline
%    \end{tabular}
%  \caption{Battery connected converter parameters.}\label{table_BAT}
%\end{table}
%%
%\begin{table}
%  \centering
%    \begin{tabular}{ | l | l | l | l |}
%    \hline
%    Parameter  & Value & Parameter  & Value \\ \hline
%    $C_7$  & 0.01 F & $L_{8}$ & 0.0033 H \\ \hline
%    $R_{07}$ & 0.01 $\Omega$  & $R_{08}$ & 0.01 $\Omega$ \\ \hline
% %   $R_7$ & 0.1 $\Omega$ & $R_{05}$ & 0.001 $\Omega$ \\ \hline
%    $R_7$ & 0.1 $\Omega$ & $C_9$  & 0.0001 F  \\
%    \hline
%    \end{tabular}
%  \caption{Supercapacitor connected converter parameters.}\label{table_SC}
%\end{table}
%
%\vspace{-1cm}
\begin{figure}%[h!]
%\vspace{-0.2cm}
  \centering
  \includegraphics[width=1\columnwidth]{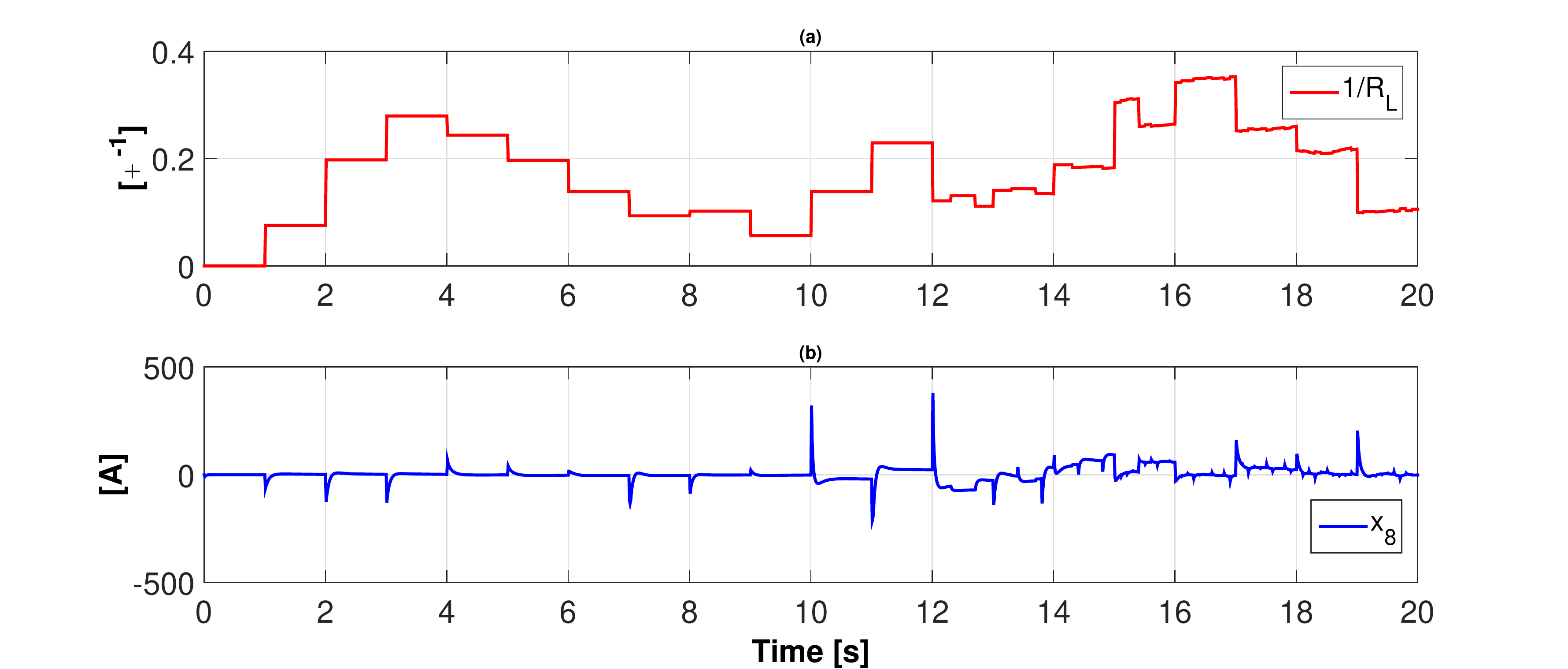}\\
 % \vspace{-0.4cm}
  \caption[The load resistance $\frac{1}{R_L}$ in (a) and the current $x_8$ in (b)]{The load resistance $\frac{1}{R_L}$ in (a) and the current $x_8$ in (b).}\label{Fig_PV_BAT_SC_currentsx8iL}
 % \vspace{-0.2cm}
\end{figure}
%\vspace{-1cm}
%
%
%\vspace{-1cm}
\begin{figure}%[h!]
%\vspace{-0.2cm}
  \centering
  \includegraphics[width=1\columnwidth]{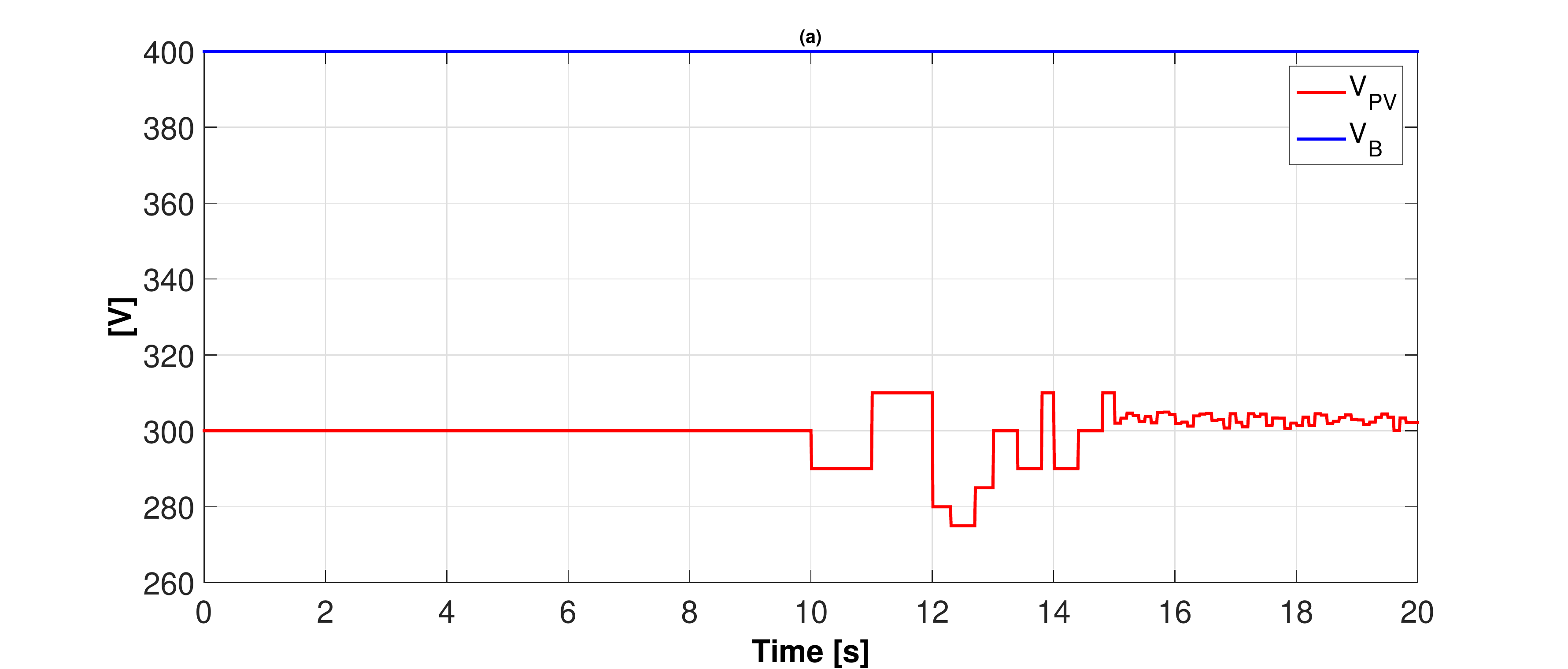}\\
 %   \vspace{-0.4cm}
  \caption[The voltages $V_{PV}$ and ${V_B}$]{The voltages of $V_{PV}$ (red line) and ${V_B}$ (blue line).}\label{Fig_PV_BAT_SC_voltagesVpvVb}
%  \vspace{-0.2cm}
\end{figure}
%\vspace{-1cm}
%
%\vspace{-1cm}
\begin{figure}%[h!]
%\vspace{-0.2cm}
  \centering
  \includegraphics[width=1\columnwidth]{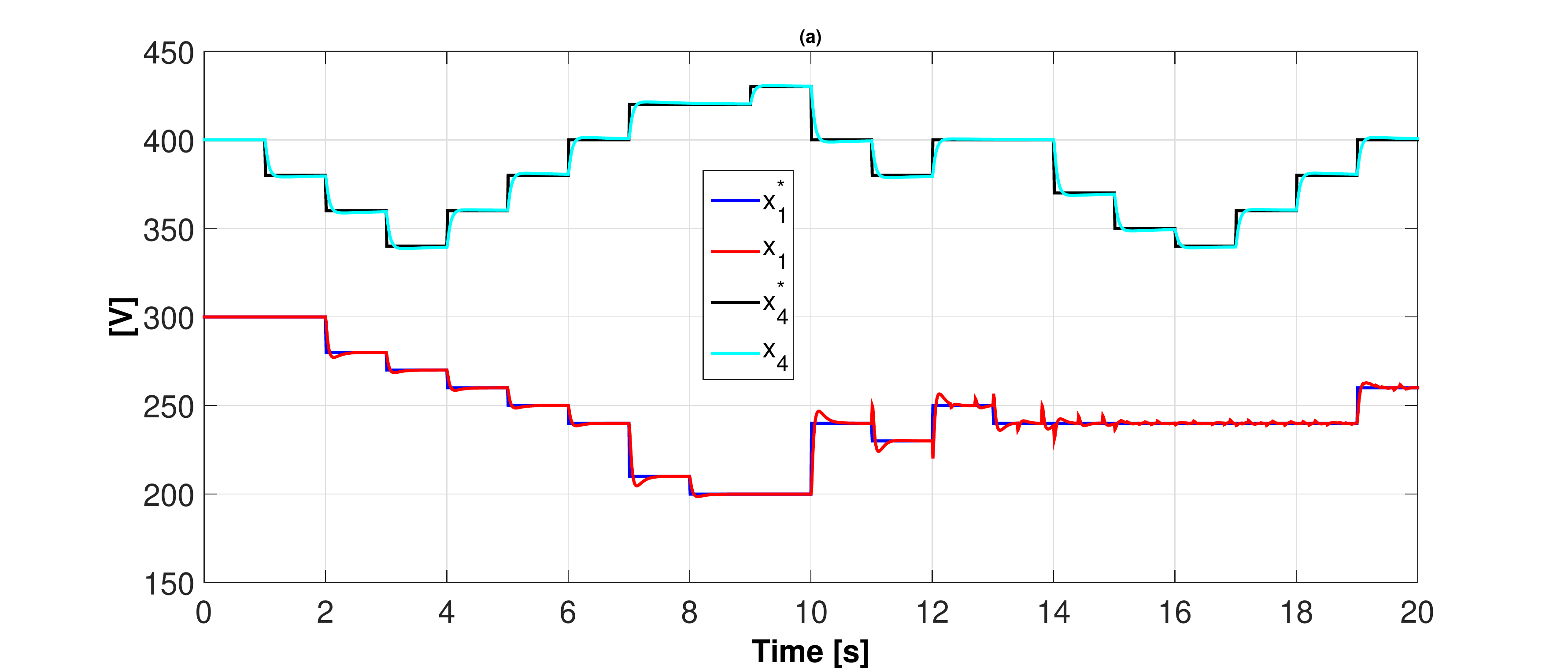}\\
    \vspace{-0.4cm}
  \caption[The voltages of ${C_1}$ and ${C_4}$]{The voltages of ${C_1}$ (red line) and ${C_4}$ (cyan line) following the desired references, the blue and black lines, respectively.}\label{Fig_PV_BAT_SC_voltagesx1x4}
%  \vspace{-0.2cm}
\end{figure}
%\vspace{-1cm}
%
%
%\vspace{-1cm}
\begin{figure}%[h!]
%\vspace{-0.2cm}
  \centering
  \includegraphics[width=1\columnwidth]{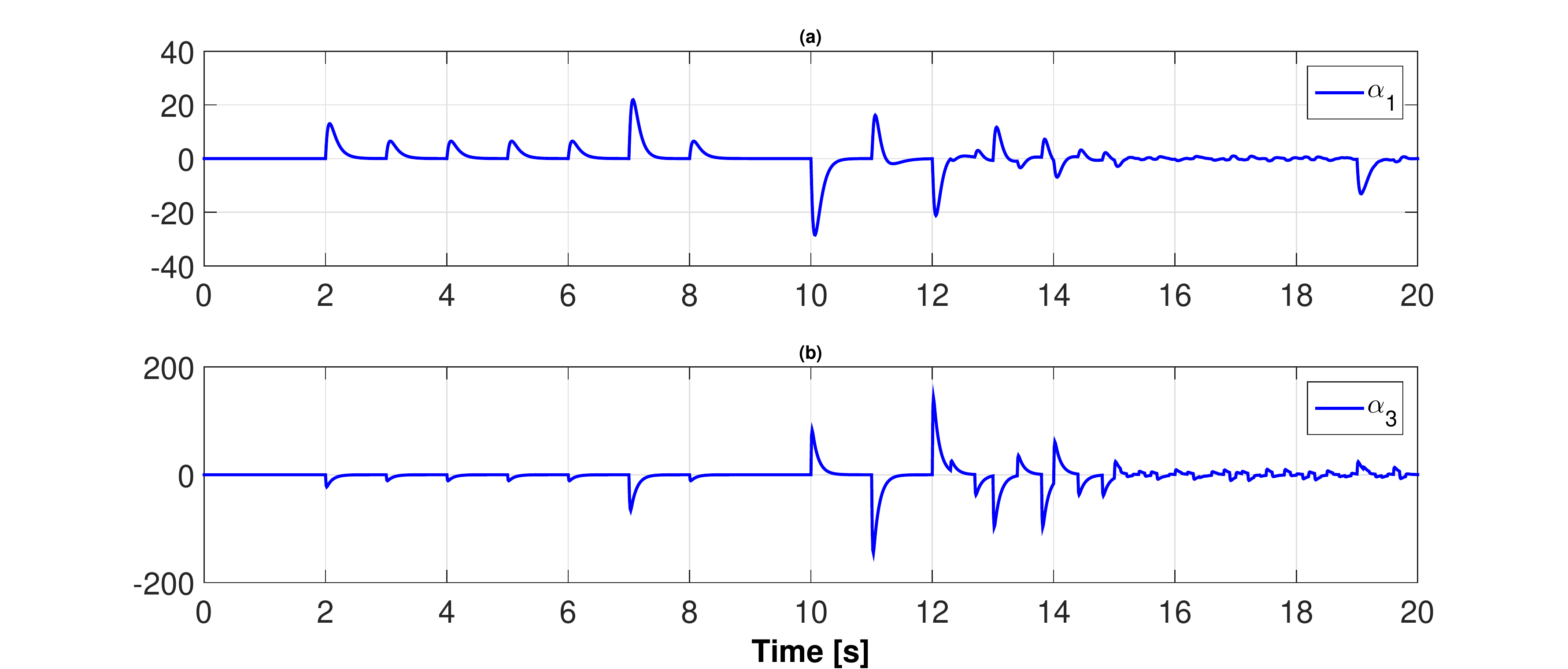}\\
%    \vspace{-0.6cm}
  \caption[The integral terms used by the control law $u_1$]{The integral terms used by the control law $u_1$: in (a) $\alpha_1$, in (b) $\alpha_3$.}\label{Fig_PV_BAT_SC_alfa13}
%    \vspace{-0.4cm}
\end{figure}
%\vspace{-1cm}
%
In accordance to the values of $\frac{1}{R_L}$, the references $x_1^*$ and $x_4^*$ are obtained for the power balance target. As depicted in Figure \ref{Fig_PV_BAT_SC_voltagesx1x4}, the $C_1$ and $C_4$ capacitor voltages reach the desired values during the considered time step. Here two different situations for the controllers are faced because the two devices need two different treatments; we need from the PV array to start providing the highest level of power as soon as possible, while the battery needs to have a smooth behaviour to preserve its life-time. The integral terms in the control action are introduced in Figures \ref{Fig_PV_BAT_SC_alfa13} and \ref{Fig_PV_BAT_SC_alfa46}; the considered eigenvalues for the systems are different because of the different targets. We note that both the charge and discharge battery situations are faced. Here we considered the voltage of the battery not to be affected by the current behaviour; indeed a constant value is used to represent it because the battery is supposed to be sized in such a way that it is not affected by current dynamics over a time of twenty seconds.
%
%\vspace{-1cm}
\begin{figure}%[h!]
%\vspace{-0.2cm}
  \centering
  \includegraphics[width=1\columnwidth]{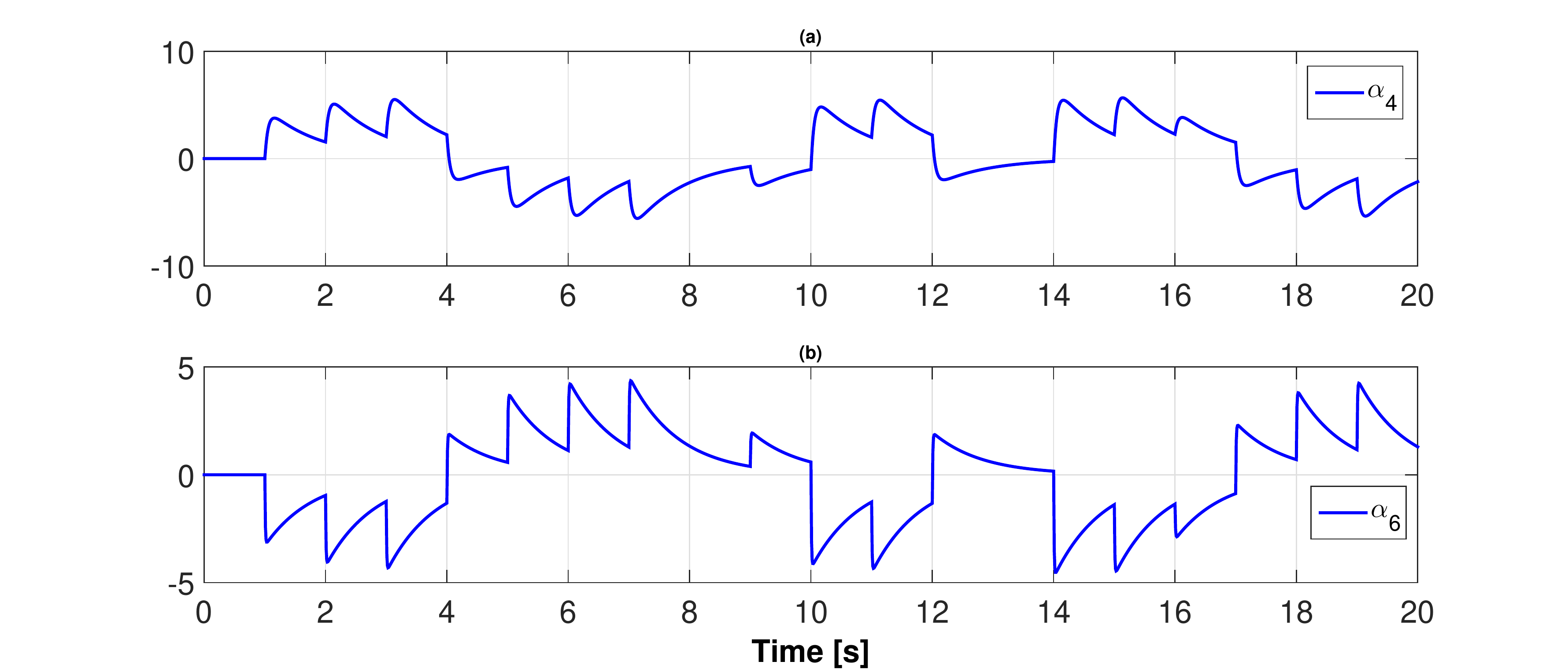}\\
 %   \vspace{-0.6cm}
  \caption[The integral terms used by the control law $u_2$]{The integral terms used by the control law $u_2$: in (a) $\alpha_4$, in (b) $\alpha_6$.}\label{Fig_PV_BAT_SC_alfa46}
 %   \vspace{-0.4cm}
\end{figure}
%\vspace{-1cm}
%
%
The resulting voltage dynamics on the grid connected capacitors, $C_2$ and $C_5$, are modified by the current flow generated by the sources; all the dynamics are stable, as shown in Figure \ref{Fig_PV_BAT_SC_voltagesx2x5}. Their evolution is influenced by the value of the DC grid voltage, which is the capacitor $C_9$: its value over time is depicted in Figure \ref{Fig_PV_BAT_SC_currentsx7x9}. To satisfy stability constraints, in response to the load variations and to the missing power coming from the PV and the battery, the voltage of the capacitor $C_7$ reacts balancing the energy variation.
Figure \ref{Fig_PV_BAT_SC_currents} describes the currents generated from the PV and the battery: these dynamics are dependent on the voltages and related to them.
%
%
%\vspace{-1cm}
\begin{figure}%[h!]
%\vspace{-0.2cm}
  \centering
  \includegraphics[width=1\columnwidth]{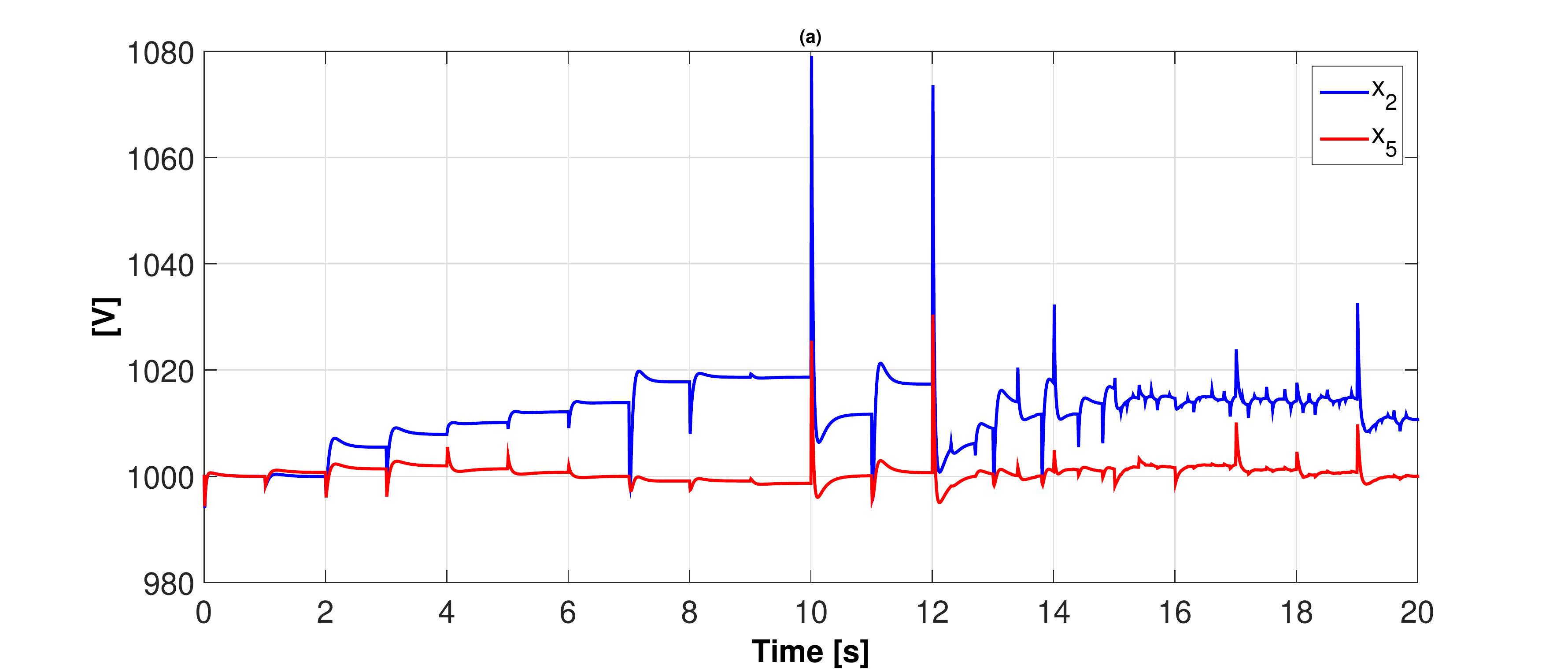}\\
 %   \vspace{-0.4cm}
  \caption[The voltages of ${C_2}$ and ${C_5}$]{The voltages of ${C_2}$ (blue line) and ${C_5}$ (red line).}\label{Fig_PV_BAT_SC_voltagesx2x5}
%  \vspace{-0.2cm}
\end{figure}
%\vspace{-1cm}
%
%\vspace{-1cm}
\begin{figure}[h!]
%\vspace{-0.2cm}
  \centering
  \includegraphics[width=1\columnwidth]{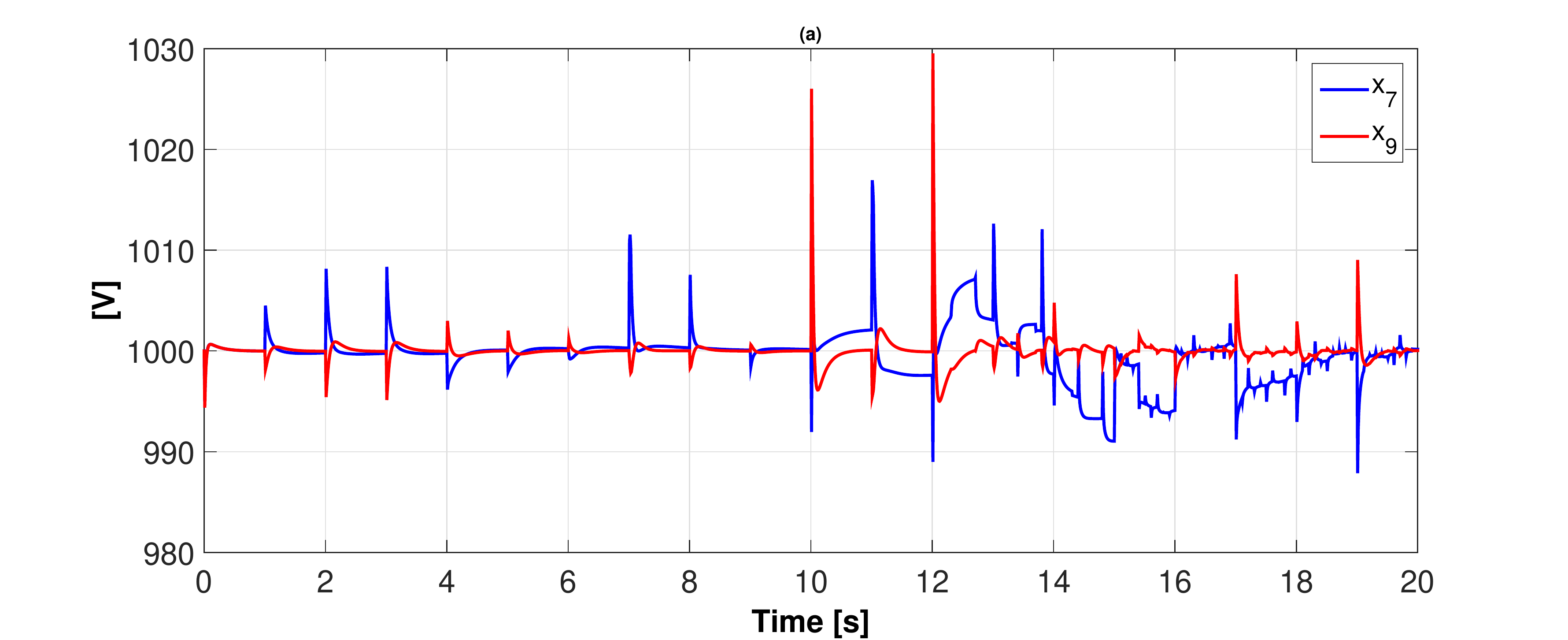}\\
 %   \vspace{-0.4cm}
  \caption[The voltages of ${C_7}$ and ${C_9}$]{The voltages of ${C_7}$ (blue line) and ${C_9}$ (red line).}\label{Fig_PV_BAT_SC_currentsx7x9}
%  \vspace{-0.2cm}
\end{figure}

%\vspace{-0.2cm}
%
%\vspace{-1cm}
\begin{figure}[h!]
%\vspace{-0.2cm}
  \centering
  \includegraphics[width=1\columnwidth]{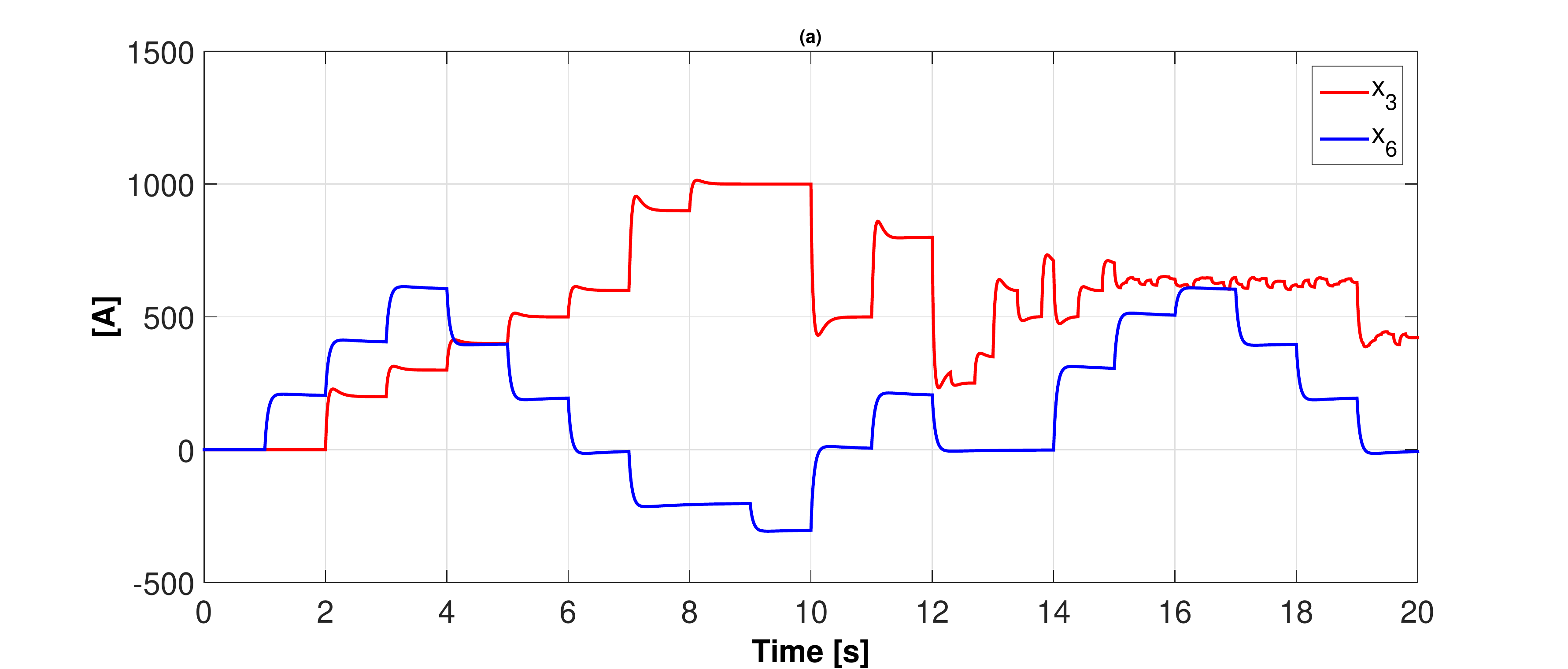}\\
 %   \vspace{-0.4cm}
  \caption[The currents $x_3$ and $x_6$]{The currents $x_3$ (red line) and $x_6$ (blue line).}\label{Fig_PV_BAT_SC_currents}
%  \vspace{-0.3cm}
\end{figure}
%\vspace{-1cm}
%
%\vspace{-1cm}
\begin{figure}[h!]
%\vspace{-0.2cm}
  \centering
  \includegraphics[width=1\columnwidth]{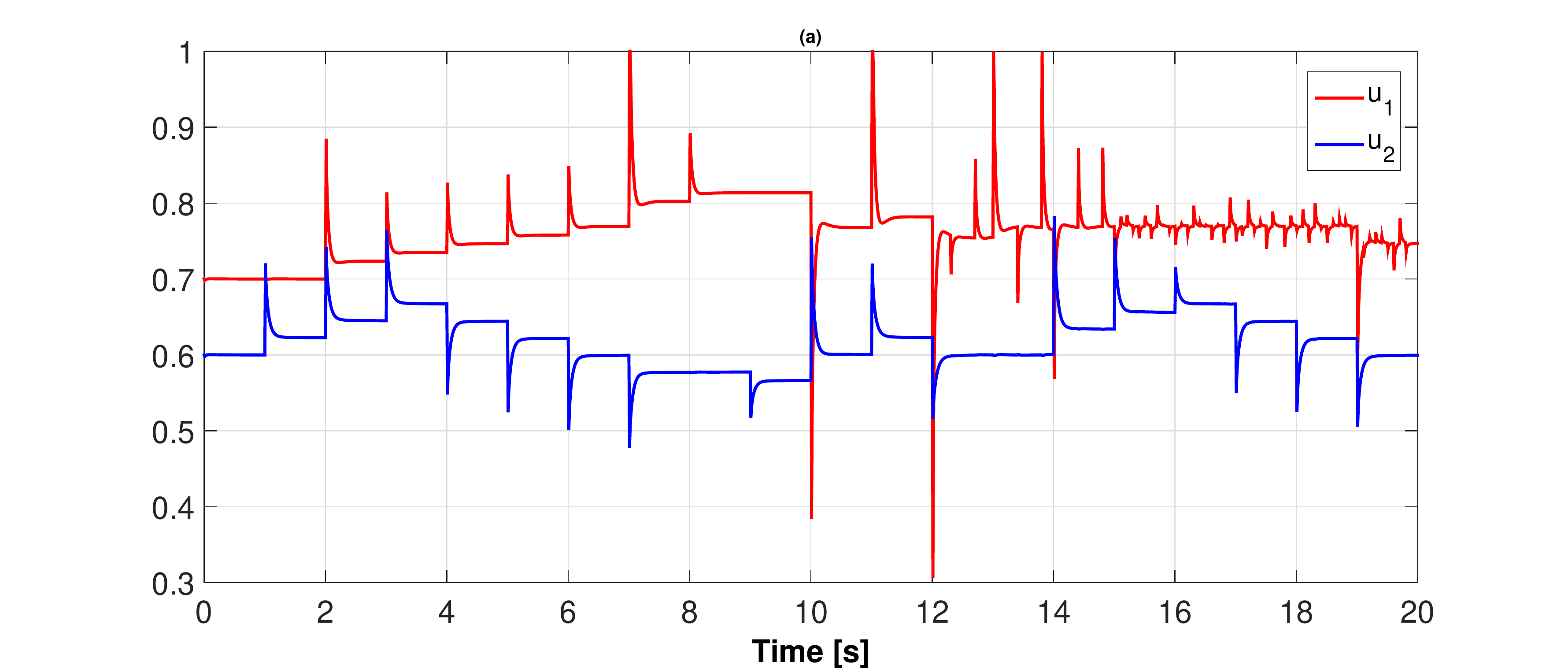}\\
%    \vspace{-0.4cm}
  \caption[The control inputs $u_1$ and $u_2$]{The control inputs $u_1$ (red line) and $u_2$ (blue line).}\label{Fig_PV_BAT_SC_controlu1u2}
%  \vspace{-0.3cm}
\end{figure}
%\vspace{-1cm}
%
Figure \ref{Fig_PV_BAT_SC_controlu1u2} introduces the generated control inputs $u_1$ and $u_2$, that are bounded
by the devices to be between zero and one, while Figure \ref{Fig_PV_BAT_SC_voltagesVpvVb} shows the voltages of the PV array and of the
battery. The control inputs are smooth except in the case of reference step variations.
The control input for the DC/DC converter connected to the supercapacitor is depicted in Figure \ref{Fig_PV_BAT_SC_controlu3}: its
variation depends on the variations of voltage $V_S$ (see Figure \ref{Fig_PV_BAT_SC_voltagesVs}) and of the load (see
Figure \ref{Fig_PV_BAT_SC_currentsx8iL}).
%
%\vspace{-1cm}
\begin{figure}%[h!]
%\vspace{-0.2cm}
  \centering
  \includegraphics[width=1\columnwidth]{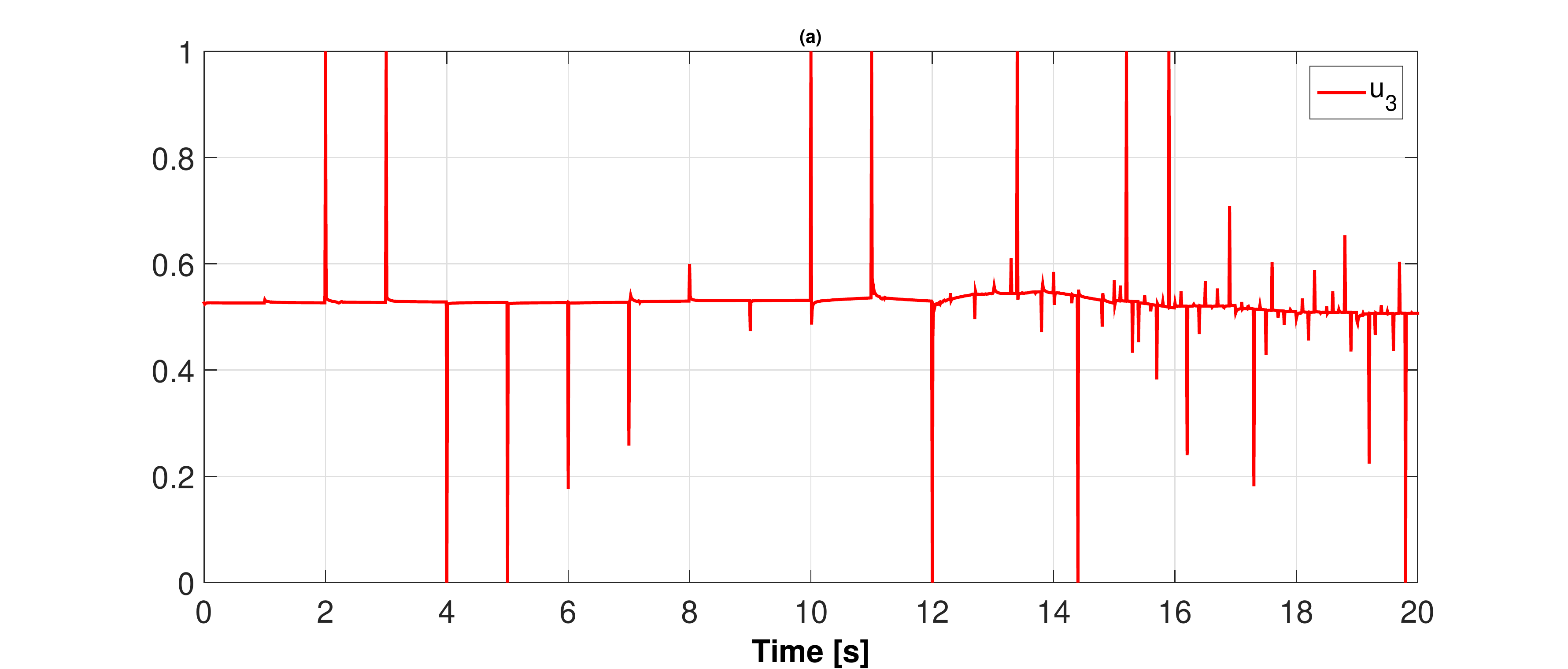}\\
 %   \vspace{-0.4cm}
  \caption{The control input $u_3$.}\label{Fig_PV_BAT_SC_controlu3}
%  \vspace{-0.3cm}
\end{figure}
%\vspace{-1cm}
%
%\vspace{-1cm}
\begin{figure}%[h!]
%\vspace{-0.2cm}
  \centering
  \includegraphics[width=1\columnwidth]{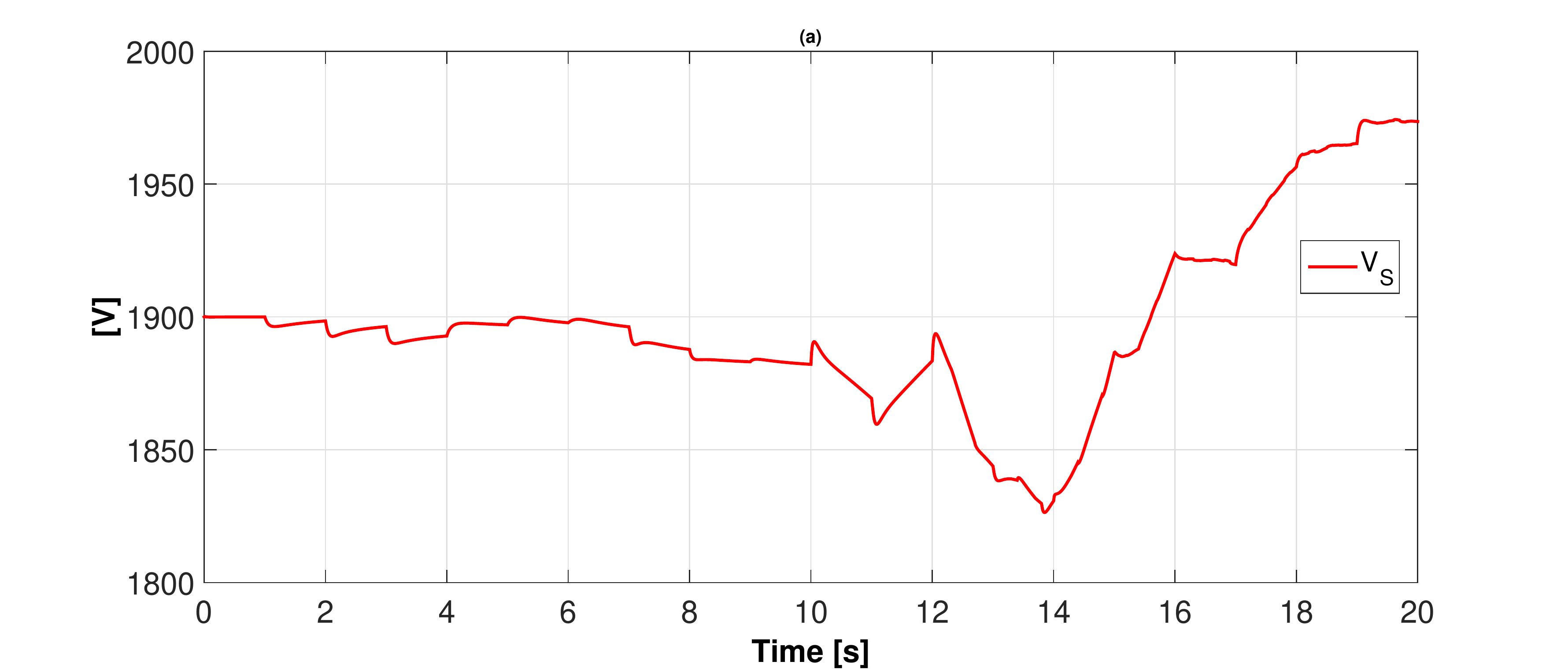}\\
%    \vspace{-0.4cm}
  \caption{The voltage of the supercapacitor.}\label{Fig_PV_BAT_SC_voltagesVs}
%  \vspace{-0.3cm}
\end{figure}
%\vspace{-1cm}
%
As results, the desired voltage for the DC microgrid is always kept (Figure \ref{Fig_PV_BAT_SC_currentsx7x9}). The developed control strategy is then shown to successfully operate in a wide range of situations: constant and time-varying load, acting of perturbation on the sources and big step variations of the references.
%
%
%\vspace{-0.6cm}
\section{CONCLUSIONS}\label{sec_conclusion}
%\vspace{-0.3cm}
In this paper a realistic DC MicroGrid composed by a PV array, two storage devices, a load and their connected devices is modeled. It is controlled in order to correctly provide a desired amount of power for feeding an uncontrolled bounded load while ensuring a desired grid voltage value. Hypotheses on the ad hoc size of the components are done to physically allow the power exchange. Stability analysis is carried out for the complete system and physical limitations are also considered. Simulations show the robustness of the adopted control action both during the transient and in steady-state operation mode in case of constant or time-varying load.

%-----------------------------------

%\vspace{-0.5cm}
%\footnotesize
%\bibliographystyle{plain}

\bibliographystyle{ieeetr} 
%\bibliography{biblio_TASE}

\bibliography{mcnbib_traffic}
%\bibliography{biblio_traffic_energy}
%\normalsize
%\bibliographystyle{plain}
%\bibliography{mcnbib}
%%\vspace{-0.2cm}

\end{document}